\newif\ifarxiv\arxivfalse
\newif\ifaccel\accelfalse
\newif\ifcolors\colorsfalse
\arxivtrue

\ifarxiv
	\documentclass[10pt, twoside, reqno]{amsart}
\else
	\RequirePackage{amsmath}
	\documentclass[]{interact}
\fi

 \usepackage[noload={subcaption,caption,algorithmic}]{myPreamble}

 \usepackage[noadjust]{cite}

\def\captionsetup#1{} 
\makeatletter
    \newcounter{ALC@unique}
\makeatother
\usepackage{algpseudocode}

\newcommandx\sw[4][{4=k}]{S_{#1}^{#2}\seq{#3^t}[t\leq{#4}]}

\def\mydot{\hspace*{0.25pt}\raisebox{0.6pt}{\scalebox{0.8}{\text{\tiny\textbullet}}}\hspace*{0.25pt}}

\def\hl#1{{\color{blue}#1}}

\newcommand{\cf}{\textit{cf. }}

\DeclareMathOperator\blkdiag{blkdiag}
\DeclareMathOperator\ri{ri}

\newcommand{\vertiii}[1]{{\left\vert\kern-0.25ex\left\vert\kern-0.25ex\left\vert #1 
    \right\vert\kern-0.25ex\right\vert\kern-0.25ex\right\vert}}

\newcommand{\TheShortTitle}{Multi-agent structured optimization with bounded communication delays}
\newcommand{\TheTitle}{Primal-dual algorithms for multi-agent structured optimization over message-passing architectures with bounded communication delays}
\newcommand{\TheShortAuthor}{P. Latafat and P. Patrinos}
\newcommand{\TheFunding}{%
	This work was supported by the Research Foundation Flanders (FWO) PhD grant 1196820N and research projects G0A0920N, G086518N and G086318N;
    Research Council KU Leuven C1 project No. C14/18/068;
    Fonds de la Recherche Scientifique -- FNRS and the Fonds Wetenschappelijk Onderzoek -- Vlaanderen under EOS project no 30468160 (SeLMA)
} 
\newcommand{\TheKeywords}{%
    Asynchronous algorithms, %
	Primal-dual algorithms, %
    Distributed optimization, %
	Message passing,
} 
\newcommand{\TheAbstract}{%
	We consider algorithms for solving structured  convex optimization problems over a network of agents with communication delays. 
    It is assumed that each agent performs its local updates by using possibly outdated information from its neighbors under the assumption that the delay with respect to each neighbor is bounded but otherwise arbitrary. The private objective of each agent is represented by the sum of two possibly nonsmooth functions, one of which is composed with a linear mapping. The global optimization problem is the aggregate of the local cost functions and a common Lipschitz-differentiable term. When the coupling between the agents is represented  only through the common function the primal-dual algorithm proposed by V\~u and Condat can be conveniently employed, while for more general structures a new algorithm is proposed. Moreover, a randomized variant is presented that allows the agents to wake up at random and independently from one another. The convergence of each of the proposed algorithms is established under different strong convexity assumptions.%
}

\ifarxiv
	\title[\TheShortTitle]{\Large\scshape\bfseries\TheTitle}
	\author[\TheShortAuthor]{%
		Puya Latafat and 
		Panagiotis Patrinos%
	}
	\thanks{%
		\TheAddressKU.
		{\tt
			\{%
				\href{mailto:puya.latafat@kuleuven.be}{puya.latafat},%
				\href{mailto:panos.patrinos@esat.kuleuven.be}{panos.patrinos}%
			\}%
			\href{mailto:puya.latafat@kuleuven.be,panos.patrinos@esat.kuleuven.be}{@esat.kuleuven.be}%
		}%
	\\
		\TheFunding
	}

\begin{document}

	\begin{abstract}
		\TheAbstract
	\end{abstract}
 
	\maketitle

\else

	\begin{document}
    \title{\TheTitle}

    \author{
    \name{Puya Latafat\thanks{Corresponding author: Puya Latafat Email: puya.latafat@kuleuven.be 
    } 
    \thanks{\TheFunding}
      \thanks{A preliminary version of this work was presented at the IEEE Conference on Decision and Control, FL, USA, 2018 \cite{Latafat2018Multiagent}.}}
    \affil{\TheAddressKU}
    }
	\maketitle

	\begin{abstract}
		\TheAbstract
	\end{abstract}


\fi


	\section{Introduction}
\label{sec:intro:RandAsyn}

In this paper we consider a class of structured optimization problems that can be represented as follows:
 \begin{equation} \label{eq:main-prob_intro:RandAsyn}
  \underset{x\in\R^n}{\minimize}\ {f}({x})+\sum_{i=1}^m\big({g_i}({x_i})+h_i(N_{i}x)\big),
  \end{equation}
  where $x=(x_1,\ldots,x_m)$, $N_i$ is a linear mapping, $h_i$, $g_i$ are proper closed convex (possibly) nonsmooth functions, $g_i$ are in addition strongly convex, and $f$ is convex, continuously differentiable with Lipschitz continuous gradient.   The goal is to solve \eqref{eq:main-prob_intro:RandAsyn} over a network of agents through local communications. Each agent is assumed to maintain its own private cost functions $g_i$ and $h_i$, while $f$ and (possibly) the linear mappings $N_i$ represent the coupling between the agents.
  In practice local communications between agents are subject to delays and/or dropouts which constitutes an important challenge addressed here. 

Most iterative algorithms for convex optimization can be written as 
\begin{equation}\label{eq:contra:RandAsyn}
  z^{k+1} = z^k - Tz^k,
\end{equation}
where the mapping $\id -T$ ($\id$ is the identity operator) has some contractive property resulting in the convergence of the sequence to a zero of $T$. In distributed optimization the goal is to devise algorithms where a group of agents/processors distributively update certain coordinates of $z$ while guaranteeing convergence to a zero of $T$. 

  There are two main computational models in distributed optimization (depicted in \cref{fig1:RandAsyn}) with a range of hybrid models in between \cite[\S 1]{Bertsekas1989Parallel}. These models are conceptually different and require different analysis. The model considered here is the local/private-memory model. Let us first describe the two models. 
 \begin{figure}
	\includegraphics[clip, trim= 2cm 1cm 0cm 5.7cm,width=\textwidth]{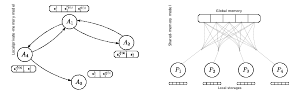}
	\caption{The two main memory models; (left) agents cooperating to perform a task, (right) processors updating a global memory
  }
  \label{fig1:RandAsyn}
\end{figure}

{\bf Shared-memory model:} 
This model is characterized by the access of all agents/processors to a shared memory. A large body of literature exists for parallel coordinate descent algorithms for this problem. Typically, coordinate descent algorithms would require a memory lock to ensure consistent reading.  Interesting recent works allow inconsistent reads \cite{Liu2015Asynchronous,Peng2016ARock}. 
In this model, for the fixed point iteration  \eqref{eq:contra:RandAsyn}, each processor reads the  global memory 
and proceeds to choose a random  coordinate $i\in\{1,\ldots,m\}$ and to perform 
\begin{equation*}
  z_i^{k+1} = z_i^k - T_i\hat{z}^k,
\end{equation*}
where $\hat{z}^k$ denotes the data loaded from the global memory to the local storage at the clock tick $k$, and $T_i$ represents the operator that updates the $i$-th coordinate. 
This form of updates are asynchronous in the sense that the processors update the global memory simultaneously resulting in possibly inconsistent  local copy $\hat{z}^k$ due to other processors modifying the  global memory during a read. 
The analysis of such algorithms would in general rely on either using the properties of the operator that updates the $i$-th coordinate when possible (coordinate-wise Lipschitz continuity in the case of the gradient \cite{Liu2015Asynchronous}), or the properties of the global operator (see \cite{Peng2016ARock} for nonexpansive operators). 
A crucial point in the convergence analysis of such methods is the fact that for a given processor, the index of the coordinate to be updated is selected at random, but no matter which coordinate is selected the same local data $\hat{z}^k$ is used for the update. 
 Let $\hat{T}_iz \coloneqq (0,\ldots,0,T_iz,0\ldots,0)$. Then, in a randomized scheme $\hat{T}_i$ can be summed over $i$:    
\begin{equation*}
  \sum_{i=1}^m \hat{T}_i\hat{z}^k = T\hat{z}^k,
\end{equation*}
allowing one to use the properties known for the global operator (see the proof of \cite[Lem. 7]{Peng2016ARock}). 
This type of argument is also used in \cite{Wu2018Decentralized} in the context of decentralized consensus optimization. See \cite{Cannelli2019Asynchronous} for a detailed discussion on the assumptions that are often imposed in this model. 
As we discuss below, the difficulty in the local-memory model is precisely due to the fact that this summation no longer holds.

{\bf Local/private-memory model:}
 In this model each agent/processor has its own private local memory. The agents can send and receive information to other agents as needed, and agent $i$ can only update $z_i$. 
This model is also referred to as \emph{message-passing model} \cite{Bertsekas1989Parallel}.  

In the absence of delay between agents, randomized block-coordinate updates may be used to develop distributed asynchronous algorithms. Such schemes would typically involve random independent activation of agents to perform their local updates, and are in this sense also referred to as asynchronous \cite{Iutzeler2013Asynchronous,Bianchi2016Coordinate,Latafat2019Randomized,Pesquet2015class}. Note that in these schemes while the agents may wake up to perform their updates at different times, the information used by each agent is assumed to be up to date, \ie,  synchronization is required.  

In accordance with the notation of the seminal work \cite[\S 7]{Bertsekas1989Parallel} we define the following local (outdated) version of the generic vector $z^k=(z_1^k,\ldots,z_m^k)$ used by agent $i$:
\begin{equation}\label{eq:outdated:RandAsyn}
z^k[i] \coloneqq \left(z_1^{\tau_1^i(k)},\ldots,z_m^{\tau_m^i(k)}\right),
\end{equation}
where $\tau_j^i(k)$ is the latest time at which the value of $z_j$ is transmitted to agent $i$ by agent $j$. 
In our setting the delay is assumed to be bounded:
\begin{ass}\label{ass:delays:RandAsyn}
  There exists an integer $B$ such that
  for all $k\geq 0$ the following holds
    $$ (\forall i,j) \quad 0\leq k-\tau_j^i(k)\leq B, \quad \textrm{and} \quad \tau^i_i(k)=k.$$

\end{ass}
The fact that each agent knows its own local variable without delay is projected in the assumption $\tau^i_i(k)=k$. This is a natural assumption and is satisfied in practice. Notice that for ease of notation we defined the complete outdated vector while in practice each agent would only keep a local copy of the coordinates that are required for its computation, see \cref{fig1:RandAsyn}. The directions of the arrows in \cref{fig1:RandAsyn} signify the nature of the coupling between two agents. For example, the arrow from $A_4$ to $A_3$ indicates that agent $A_3$ requires $z_4$ for its computation. Such a relation between agents is dependent on the formulation and the nature of coupling between agents. For instance, in \eqref{eq:main-prob_intro:RandAsyn} the coupling is represented through $f$ and possibly $N_i$. As we shall see in \cref{sec:example:RandAsyn} the coupling through $f$ may be one sided since agent $i$ may require information from agent $j$ for computing $\nabla_i f$ (the partial derivative of $f$ with respect to $i$-th coordinate) without the reverse relation being true. 

In summary, each agent controls only one block of coordinates and updates according to 
\begin{equation*}
	z_i^{k+1}  =z_i^k - T_iz^k[i],
\end{equation*}
the result of which will be sent (possibly with different delay) to the agents that require it in their computations. 
The difficulty in this model comes from the impossibility of summing $T_iz^k[i]$ over all $i$ given that $z^k[i]$ is different for each $i$. 


  In addition to the above described delay, the \emph{partially asynchronous} (PA) protocol considered in \cite[\S 7]{Bertsekas1989Parallel} involves a second assumption: each agent must perform an update at least once during any time interval of length $B$. In \cite[\S 7.5]{Bertsekas1989Parallel} a PA variant of the gradient method  is studied. This analysis is further extended to the projected-gradient method in the convex case.  In \cite{Paul1991Rate} a periodic linear convergence rate is established for the projected-gradient method. The recent work \cite{Zhou2018Distributed} extends this analysis to the proximal-gradient method. 

  The aforementioned primal methods are not well equipped for problems with more complex structures as in \eqref{eq:main-prob_intro:RandAsyn}.  An efficient way to tackle such problems is to employ a class of first-order methods, referred to as primal-dual algorithms. This approach leads to \emph{fully split} algorithms eliminating the need for inverting matrices or solving inner loops. Developing PA schemes for primal-dual algorithms is not addressed here and remains a challenge. 
It it worth noting that \cite{Hale2017Asynchronous} considers a primal-dual framework under a different asynchronous protocol where the primal variables follow a totally asynchronous model \cite[\S 6]{Bertsekas1989Parallel}. However, the dual variables are required to be synchronized across agents.

It is worth noting that finite sum minimization over graphs is another popular problem that has been considered by many authors \cite{Nedic2009Distributed,Duchi2012Dual,Shi2015Proximal,Lobel2011Distributed,Johansson2010Randomized,Latafat2016primaldual}. Several asynchronous algorithms have been studied for this popular problem over master-worker architectures \cite{Agarwal2011Distributed,Aytekin2016Analysis,Feyzmahdavian2015asynchronous,Chang2016Asynchronous,Zhang2014Asynchronous}. Moreover, asynchronous subgradient type methods have been studied extensively for finite sum problems \cite{Nedic2001381Distributed,Tsianos2012Distributed,Terelius2011Decentralized,Wang2015Cooperative,Lin2016Distributed}. 
In this work general optimization problem \eqref{eq:main-prob_intro:RandAsyn} is considered. This framework can be used to  develop asynchronous distributed proximal algorithms for general finite sum minimization of the form
  \begin{equation*}
    \minimize_{x\in\R^n}\,\sum_{i=1}^{m} \varPsi_i(x)+ \varTheta_{i}(x)+\varPhi_{i}(C_{i}x),
  \end{equation*}
  where  $\varPhi_i$ is smooth, $\varTheta_i$ and $\varPhi_i$ are (possibly) nonsmooth extended-real-valued and $C_i$ is a linear mapping. This problem can be reformulated as \eqref{eq:main-prob_intro:RandAsyn} using a consensus reformulation according to the communication graph. However, this is deferred to future work.

\subsection{ Motivating examples}\label{sec:MotExm:RandAsyn}

Consider the regularized logistic regression problem 
\begin{equation}\label{eq:LERM:RandAsyn}
  \minimize_{w\in\R^n}\, \sum_{i=1}^m \sum_{j\in \mathcal{I}_i} \log\left(1+ \exp\left(-y_{(j)}\langle x_{(j)}, w\rangle\right)\right) + \lambda \|w\|^2,
\end{equation}
where $w=(w_1,\ldots,w_m)\in\R^n$ is the regression vector, $\lambda$ is a positive constant, and the data is distributed between $m$ machines; the pair $(x_{(j)},y_{(j)})_{j\in\mathcal{I}_i}$ represents the data stored at the $i$-th machine. 
 The goal is to solve the global minimization via local communications which may be subject to communication delays.  Clearly \eqref{eq:LERM:RandAsyn} fits into the form of \eqref{eq:main-prob_intro:RandAsyn}: let $g$  represent the separable regularizer,  $h_i(v)=\sum_{j\in \mathcal{I}_i}\log((1+ \exp(-y_{(j)}v_j))$ the loss function, 
  $f\equiv 0$ and the rows of $N_i$ consisting of $x^\top_{(j)}$ for $j\in\mathcal{I}_i$. 
In this formulation the coupling is through the linear terms (\cf \cref{sec:generalL:RandAsyn}). 
 Distributed elastic net problem is another such example with $h_i$ representing the squared loss, $N_i$ the locally stored data, $g$ the elastic net regularizer and $f\equiv 0$. Note that in both examples $g_i$ is strongly convex and $h_i$ is continuously differentiable with Lipschitz continuous gradient, satisfying the requirements of \cref{sec:generalL:RandAsyn} for the cost functions.

Another notable example is the problem of formation control \cite{Raffard2004Distributed}, where each agent (vehicle) has its own private dynamics and cost function and the goal is to achieve a specific formation while communicating  only with a selected number of agents. Let $w_i=({\xi}_i,v_i)$ where ${\xi}_i$ and $v_i$ denote the local state and input sequences. 
The location of agent $i$ is given by $y_i=C {\xi}_i$ and the set of 
its neighbors is denoted by $\mathcal{A}_i$. The linear dynamics of each agent over a control horizon is represented by the constraints $E_i w_i = b_i$. 
 In order to enforce a formation between agents $i$ and $j$ the quadratic cost function $\|C(\xi_i- \xi_j) - d_{ij}\|^2$ is used where $d_{ij}$ is the target relative distance between them (refer to \cite{Raffard2004Distributed} for details). Hence, the formation control problem is formulated as the following constrained minimization:
    \begin{subequations}\label{eq:formation}
   \begin{align}
    {\minimize}&\sum_{i=1}^m \tfrac{\lambda_i}{2}\sum_{j\in\mathcal{A}_i}\|C(\xi_i- \xi_j) - d_{ij}\|^2 + \tfrac{1}{2}\sum_{i=1}^mw_i^\top Q_i w_i 
      \\\stt&\, E_iw_i=b_i,\quad w_i\in \mathcal{W}_i, \quad i=1,\ldots,m  
  \end{align}
   \end{subequations}
  This problem can be easily cast in the form of \eqref{eq:main-prob_intro:RandAsyn} by setting $f$ equal to the first term, $g_i$ equal to the quadratic local cost, while $h_i\circ L_i$ captures the dynamics and input and state constraints (see \Cref{sec:Formation} for more details). Therefore, the objective is to enforce a formation between agents by solving this optimization problem in presence of communication delays by allowing the agents to use outdated information. Notice that in this case the coupling between agents is enforced only through $f$. 
   This special case of \eqref{eq:main-prob_intro:RandAsyn} is studied in \cref{sec:BlockL:RandAsyn}. 
   


\subsection{Main contributions} 
{\def\item{\par\noindent{\small\textbullet}\hspace*{4pt}}

  \item 
  To the best of our knowledge this is the first work that considers the delay described in \eqref{eq:outdated:RandAsyn} in a message-passing model for primal-dual algorithms. Unlike primal methods (gradient or proximal-gradient), the proposed algorithms are applicable to problems with complex structures as in \eqref{eq:main-prob_intro:RandAsyn} without the need to  solve inner loops or to invert matrices. 

\item The analysis of \cite{Bertsekas1989Parallel,Paul1991Rate,Zhou2018Distributed} rely on the use of the cost function as the Lyapunov function. In contrast, we show that quasi-Fej\'er monotonicity is an effective tool in the analysis of bounded delays in our setting. While this paper focuses on two particular primal-dual algorithms, a similar analysis should be applicable to others such as those proposed in \cite{Combettes2012Primaldual,BricenoArias2011monotone,Drori2015simple,Latafat2017Asymmetric,Latafat2019Randomized,Latafat2018PrimalDual}. 


\item Two primal-dual algorithms are presented: (i) when the coupling between agents is enforced only through $f$, the algorithm of \cite{Condat2013primaldual,Vu2013splitting} is considered (\cf \cref{sec:BlockL:RandAsyn}), (ii) when the coupling is through $f$ and the linear term, a new modified algorithm is developed (\cf \cref{sec:generalL:RandAsyn}). 
In addition, linear convergence rates are established with explicit convergence factors.


  \item  In \cref{sec:rand:RandAsyn} an asynchronous protocol is proposed; at every iteration agents are activated at random, and independently from one another, while performing their updates using outdated information. In practice, random activation can model the discrepancies in the speed of different agents.}

\subsection{Notation and Preliminaries}  
Throughout, $\R^n$ is the $n$-dimensional Euclidean space with inner product $\langle\cdot,\cdot\rangle$ and induced norm $\|\cdot\|$. For a positive definite matrix $P$ we define the scalar product $\langle x,y \rangle_P=\langle x,Py \rangle$ and the induced norm $\|x\|_P=\sqrt{\langle x,x\rangle_P}$. 

For a set $C$, we denote its relative interior by $\ri C$. Let $q:\R^n\to\Rinf\coloneqq\R\cup\{+\infty\}$ be a proper closed convex function. Its domain is denoted by $\dom q$. 
Its subdifferential is the set-valued operator $\partial q: \R^n\rightrightarrows\R^n$ 
$$
\partial q(x)=\{y\in\R^n\mid\forall z\in\R^{n},\,\langle z-x,y\rangle+f(x)\leq f(z)\}.
$$
 For a positive scalar $\rho$ the \emph{proximal map} associated with $q$ is the single-valued mapping defined by
 \begin{align*}
 \prox_{\rho q}(x)&\coloneqq 
 \argmin_{z\in\R^n}\{ q(z) + \tfrac{1}{2\rho}\|x-z\|^2\}.
 \end{align*}
 The \emph{Fenchel conjugate} of $q$, denoted by $q^*$, is defined as 
$q^*(v)\coloneqq \sup_{x\in\R^n}\{ \langle v,x\rangle-q(x)\}.$ 
The function $q$ is said to be $\mu$-convex  with $\mu\geq 0$ if $q(x)- \tfrac{\mu}{2}\|x\|^2$ is convex. 

 A sequence $\seq{w^k}$ is said to be \emph{quasi-Fej\'er} monotone with respect to a nonempty set $\mathcal{U}\subseteq \R^n$ if for all $v\in\mathcal{U}$, there exists a summable nonnegative sequence $\seq{\varepsilon^k}$ such that for all $k\in\N$ 
\begin{equation*}
  \|w^{k+1}-v\|^2\leq \|w^{k}-v\|^2 + \varepsilon^k. 
\end{equation*}
This definiton is referred to as type III quasi-Fej\'er monotonicity in \cite{Combettes2001QuasiFejerian}. 
Moreover, given a postive definite matrix $P$, we say that a sequence is $P$-quasi-Fej\'er monotone  with respect to $\mathcal{U}\subseteq \R^n$ if it is quasi-Fej\'er monotone with respect to $\mathcal{U}$ in the space equipped with $\langle\cdot,\cdot\rangle_P$. Finally, the positive part of $x\in\R$ is denoted by $[x]_+\coloneqq\max\{x,0\}$.

 Let $l$ and $d$ be two nonnegative scalars. For a given sequence $\seq{w^t}[t\in\N]$ we define the following for simplicity of notation:
\begin{equation*}
  \sw{l}{d}{w} \coloneqq 
      \sum_{\tau=[k-B+1-l]_+}^{\smash{k-d}} \|w^{\tau+1}-w^\tau\|^2.
\end{equation*}
Summing $\sw{l}{d}{w}$ over $k$ from $0$ to $p>0$ and noting that each term is repeated at most $B+l-d$ times we obtain:   
\begin{align}  
    \sum_{k=0}^p \sw{l}{d}{w} {}&{}\leq (B+l-d) \sum_{\mathclap{k=[1-B-l]_+}}^{p-d} \|w^{k+1}-w^k\|^2\leq (B+l-d) \sum_{\mathclap{k=0}}^{\smash{p}} \|w^{k+1}-w^k\|^2.
     \label{eq:zeroton:RandAsyn} 
  \end{align}
  This inequality plays an important role in our convergence analysis.

 \section{Problem setup} \label{sec:setup}
   \label{sec:example:RandAsyn}

   
   Throughout this paper the primal and  dual vectors, denoted $x$ and $u$, are assumed to be composed of $m$ blocks as follows 
   \begin{equation*}
   x=\left(x_1,\ldots,x_m\right)\in\R^n, \; 
   u=\left(u_1,\ldots,u_m\right)\in\R^r, 
   \end{equation*}
   where $x_i\in\R^{n_i}$ and $u_i\in\R^{r_i}$. Moreover, we denote the stacked primal and dual variable as $z=(x,u)$. 
   
   Consider a linear mapping  $L: \R^{n} \to \R^r$ that is partitioned as follows:
   \begin{equation}\label{eq:defL:RandAsyn}
     L=\begin{pmatrix}
   L_{11} & \cdots & L_{1m}\\
   
   \vdots& \ddots & \vdots\\
   L_{m1} & \cdots & L_{mm}
   \end{pmatrix},
   \end{equation}
   where $L_{ij}: \R^{n_i}\to \R^{r_j}$. Furthermore, the $i$-th (block) row of $L$ is denoted by $L_{i\mydot}: \R^{n}\to \R^{r_i}$ and the $i$-th (block)  column  by $L_{\mydot i}: \R^{n_i}\to \R^r$, \ie,    
   $$
   L=\begin{pmatrix}
   L_{1\mydot}\\
   \vdots\\
   L_{m\mydot}
   \end{pmatrix}= \begin{pmatrix}
   L_{\mydot 1}& \cdots
   &
   L_{\mydot m}
   \end{pmatrix}.
   $$
   The following holds
   \begin{equation}\label{eq:TheObvs:RandAsyn}
     \langle Lx,u\rangle = \sum_{i=1}^m\langle L_{i\mydot}x,u_i\rangle =\sum_{i=1}^m\langle x_i,L_{\mydot i}^\top u\rangle.
   \end{equation}
   
   Consider the structured optimization problem \eqref{eq:main-prob_intro:RandAsyn} where the linear mapping $N_i$ has been replaced by $L_{i\mydot}$ defined above in order to clarify the structure of the mapping:
     \begin{equation} \label{eq:main-prob:RandAsyn}
     \underset{x\in\R^n}{\minimize}\ {f}({x})+\sum_{i=1}^m\big({g_i}({x_i})+h_i(L_{i\mydot}x)\big).
     \end{equation}
    The cost functions $g_i$ and $h_i\circ L_{i\mydot}$ are private functions belonging to agent $i$. The coupling between agents is through the smooth term $f$ and the linear term $L_{i\mydot}x$. An agent $i$ is assumed to have access to the information required for its computation, be it outdated, \cf \Cref{Alg:PA-VU-Diagonal:RandAsyn,Alg:PA-VU:RandAsyn}. 
   
   Let the following assumptions hold
   \begin{ass} \label{ass:optProb:RandAsyn}
   	\hspace{0cm}
   	\begin{enumerate}
   		\item \label{ass:optProb-1:RandAsyn} For $i=1,\ldots,m$, $h_i:\R^{r_i}\to\Rinf$ is proper closed convex function, and $L_{i\mydot}:\R^n \to \R^{r_i}$ is a linear mapping;
   		\item \label{ass:DiagL-1:RandAsyn} (strong convexity)  For $i=1,\ldots,m$, $g_i:\R^{n_i}\to\Rinf$ is proper closed $\mu_g^i$-strongly convex for some $\mu_g^i>0$;
   		\item \label{ass:optProb-2:RandAsyn}  ${f}:\R^{n}\to\R$ is  convex, continuously differentiable, and for some $\beta\in[0,\infty)$, $\nabla f$ is $\beta$-Lipschitz continuous:   
   		\begin{equation*}
   		\|\nabla f(x)-\nabla f(x^\prime) \|\leq \beta\|x-x^\prime\|,\quad \forall x,x^\prime\in \R^n. 
   		\end{equation*} 
   		\item \label{ass:optProb-3:RandAsyn} For every $i=1,\ldots,m$ there exists a nonnegative constant $\bar{\beta}_i$ such that for all $x,x^\prime\in\R^n$ satisfying $x_i=x^\prime_i$: 
   		\begin{equation}\label{eq:Lipz:RandAsyn}
   		\|\nabla_i f(x)-\nabla_i f(x^\prime) \|\leq \bar{\beta}_i\|x-x^\prime\|. 
   		\end{equation}  
   		\item \label{ass:optProb-5:RandAsyn} The set of solutions to \eqref{eq:main-prob:RandAsyn} is nonempty. Moreover, there exists $x_i\in\ri \dom g_i$, for $i=1,\ldots,m$ such that $L_{j\mydot}x\in\ri \dom h_j$, for $j=1,\ldots,m$.  
   	\end{enumerate}
   \end{ass}
   \Cref{ass:optProb-3:RandAsyn} quantifies the strength of the coupling (through $f$) between agents \cite[\S 7.5]{Bertsekas1989Parallel}. In particular, if $f$ is separable, \ie, $f(x)=\sum_{i=1}^m f_i(x_i)$, then there is no coupling and $\bar{\beta}_i=0$. 
   
   Problem \eqref{eq:main-prob:RandAsyn} can be compactly represented as
    \begin{equation*} 
     \underset{x\in\R^n}{\minimize}\ {f}({x})+{g}({x})+h(Lx),
     \end{equation*}
     where $g(x)=\sum_{i=1}^m g_i(x_i)$, $h(u)= \sum_{i=1}^m h_i(u_i)$, and $L$ is as in \eqref{eq:defL:RandAsyn}. The dual problem is given by 
   \begin{equation*} 
     \minimize_{u\in\R^r} (g + f)^*(-L^\top u)+ h^*(u).
   \end{equation*}
   By \Cref{ass:DiagL-1:RandAsyn} the set of solutions to \eqref{eq:main-prob:RandAsyn} is nonempty and unique. Under the constraint qualification of \Cref{ass:optProb-5:RandAsyn}, the set of solutions to the dual problem is nonempty (not necessarily a singleton) and the duality gap is zero \cite[Cor. 31.2.1]{Rockafellar1970Convex}. Furthermore, $x^\star$ is a primal solution and $u^\star$ is a dual solution if and only if the pair $(x^\star,u^\star)$ satisfies
      \begin{equation} \label{eq:primal-dual:RandAsyn}
      \begin{cases}
      0 \in\partial g(x^\star)+\nabla f(x^\star)+L^{\top}u^\star
   , & \ \ \ \ \ \ \\ 0 \in\partial h^{*}(u^\star)-Lx^\star
   .& \ \ \ \ \ \ 
      \end{cases} 
      \end{equation}
    Such a point is called a primal-dual solution and the set of all primal-dual solutions is denoted by $\mathcal{S}$. 
    
   
   For each agent $i\in\{1,\ldots,m\}$ define positive stepsizes $\gamma_i$, $\sigma_i$  associated with the primal and the dual variables, respectively.
   Let us also define the following parameters
   \begin{align*}
       \Gamma    \coloneqq \blkdiag\left(\gamma_1 \I_{n_1},\ldots,\gamma_m\I_{n_m}\right),\quad
       \Sigma   \coloneqq \blkdiag\left(\sigma_1\I_{r_1},\ldots,\sigma_m\I_{r_m}\right),
   \end{align*}
   and 
   \begin{align}\label{eq:D:RandAsyn}
       D & = \blkdiag(\Gamma^{-1},\Sigma^{-1}), \quad 
       \bar{\beta} \coloneqq \left(\bar{\beta}_1,\ldots,\bar{\beta}_m\right).
   \end{align}

   The algorithm of V\~u and Condat \cite{Vu2013splitting,Condat2013primaldual} for  solving \eqref{eq:main-prob:RandAsyn} is given by the following updates for agent $i$ at iteration $k$: 
   \begin{subequations}\label{eq:VuCondat:RandAsyn}
   	\begin{align}
   	x_{i}^{k+1}	&=\prox_{\gamma_i g_{i}}\left(x_{i}^{k}-\gamma_i L_{\mydot i}^{\top}u^{k}-\gamma_i\nabla_{i}f(x^{k})\right) \\
   	u_{i}^{k+1} &=\prox_{\sigma_i h_i^{*}}\left(u_{i}^{k}+\sigma_i L_{i\mydot}(2x^{k+1}-x^{k})\right). \label{eq:VuCondat-b:RandAsyn}
   \end{align}
   \end{subequations}
   In a synchronous implementation of the algorithm, each agent requires the latest variables $x^k$, $x^{k+1}$ and $u^k$ in the above updates, which may not be available due to communication delays. In the case when $L$ is  block-diagonal the coupling between agents is enforced only through the smooth function $f$ (in \eqref{eq:VuCondat:RandAsyn} agent $i$ requires the primal variables that are required for computing $\nabla _i f$). 
   We refer to this type of coupling as \emph{partial coupling}. In \Cref{sec:BlockL:RandAsyn} the updates in \eqref{eq:VuCondat:RandAsyn} are considered for the case of partial coupling (\cf \cref{Alg:PA-VU:RandAsyn}).
   
   More generally when  $L$ is not block-diagonal, the coupling between agents is enacted through  the linear mapping $L$ (the linear operations $L_{\mydot i}^\top$ and $L_{i \mydot}$ in \eqref{eq:VuCondat-b:RandAsyn} require additional communication between agents) and possibly the smooth function $f$. We refer to this type of coupling as \emph{total coupling}. 
   This  case is considered in 
   \Cref{sec:generalL:RandAsyn} where an \emph{Arrow-Hurwicz-Uzawa} type \cite{Arrow1958Studies} (referred hereafter as AHU-type) primal-dual algorithm  is proposed in place of \eqref{eq:VuCondat:RandAsyn}. The synchronous iterations of the AHU-type algorithm for agent $i$ at iteration $k$ is given by: 
   \begin{subequations} \label{Alg:AHU:RandAsyn}
   		\begin{align}
   	x_{i}^{k+1}	&=\prox_{\gamma_i g_{i}}\left(x_{i}^{k}-\gamma_i L_{\mydot i}^{\top}u^{k}-\gamma_i\nabla_{i}f(x^{k})\right) \\
   	u_{i}^{k+1} &=\prox_{\sigma_i h_i^{*}}\left(u_{i}^{k}+\sigma_i L_{i\mydot}x^{k}\right). 
   \end{align}
   \end{subequations}
    Differently from \eqref{eq:VuCondat:RandAsyn}, in the dual update linear operator is applied to $x^k$ in place of $2x^{k+1}-x^k$. When operating under the bounded delay assumption, the AHU-type primal-dual algorithm, \eqref{Alg:AHU:RandAsyn}, allows for larger stepsizes compared to \eqref{eq:VuCondat:RandAsyn}.
   
   It is worth noting that \eqref{Alg:AHU:RandAsyn} can be seen as a forward-backward iteration:
   \begin{equation*}
   	z^{k+1} = (D + T_2)^{-1}(D - T_1)z^k,
   \end{equation*}
   where $T_2=(\partial g,\partial h^*)$,  $T_1: (x,u) \mapsto (\nabla f(x)+L^\top u,-Lx)$, and $D$ as  defined in \eqref{eq:D:RandAsyn}. 
     Even in the synchronous case this algorithm is not in general convergent. The convergence may be established when $g$ and $h^*$ are strongly convex \cite[Assumption A]{Chen1997Convergence}. 
   Moreover, when $g$ is the indicator of a set and $h$ is the support of a set, the AHU-type algorithm resembles another primal-dual AHU-type algorithm considered in \cite{Nedic2009Subgradient,Hale2017Asynchronous} for solving saddle-point problems. 
\section{The case of partial coupling} \label{sec:BlockL:RandAsyn}

   
   Throughout this section we consider the optimization problem \eqref{eq:main-prob:RandAsyn} with partial coupling (when $L$ has a block-diagonal structure).  
    In this case, the coupling between agents is enacted only through the smooth function $f$ (and not through $L$). 
   The example of formation control in \Cref{sec:MotExm:RandAsyn} can be cast in this form.  
   
   Under this setting problem \eqref{eq:main-prob:RandAsyn} becomes
     \begin{equation*} 
     \underset{x\in\R^n}{\minimize}\ {f}({x})+\sum_{i=1}^m\big(g_i(x_i)+h_i(L_{ii}x_i)\big),
     \end{equation*}
     where $L_{ii}$ is the $i$-th diagonal block of $L$, see \eqref{eq:defL:RandAsyn}. 
      In order to solve this problem with the iterates in \eqref{eq:VuCondat:RandAsyn}, agent $i$ must receive those $x_j$'s that are required for the computation of $\nabla_i f$ and all other operations are local.    Let us define two sets of indices: those that are required to send their variables to $i$:  
   \begin{align*}
   	\mathcal{N}_i^{\rm in} \coloneqq \{j \mid \nabla_if \;\textrm{depends on}\; x_j\}, 
   \end{align*}
   and those that $i$ must send $x_i$ to as $\mathcal{N}_i^{\rm out}\coloneqq \{j \mid i \in \mathcal{N}_j^{\rm in}\}$.

    \Cref{Alg:PA-VU-Diagonal:RandAsyn} summarizes the proposed scheme. At every iteration each agent $i$ performs the updates described in \eqref{eq:VuCondat:RandAsyn} using the last  information it has received from agents $j\in\mathcal{N}_i^{\rm in}$. It then transmits the updated $x_i^{k+1}$ to the agents that require it (possibly with different delay). Note that $x^k[i]$ was defined as the outdated version of the full vector $x^k$ for simplicity of notation, and in practical implementation it would only involve the coordinates that are required for the computation of $\nabla_i f$.   
    
       \newcommand{\INDSTATE}[1][1]{\State\hspace{-2pt}}
       \begin{algorithm}[H]
         \caption{V\~u-Condat algorithm with bounded delays}
         \label{Alg:PA-VU-Diagonal:RandAsyn}
         \begin{algorithmic}[1]
         \algnotext{EndFor}
           \item[\textbf{Initialize:}] $x_i^0\in\R^{n_i}$, $u_i^0\in\R^{r_i}$ for  $i\in\{1,\ldots,m\}$.
           \item[\bf For] $k=0,1,\ldots$ {\bf do}
             \item[\hspace*{13pt}\bf For] {each agent $i=1,\ldots,m$} {\bf do}
                     \Statex\hspace*{-7pt} {\it\small\% Local updates}
   
           \Statex \hspace*{11pt} {\small\sl\color{black!80}perform the local updates using the last received  information, \ie, the locally  stored \linebreak
           \hspace*{12pt} vector  $x^k[i]$ as  defined in \eqref{eq:outdated:RandAsyn}}:  
           \State\hspace*{\algorithmicindent}\hspace*{-1pt}$x_{i}^{k+1}  =\prox_{\fillwidthof[l]{\sigma_i h_i}{\gamma_i g_{i}}}\left(x_{i}^{k}-\gamma_i L_{ii}^{\top}u_i^{k}-\gamma_i\nabla_{i}f(x^{k}[i])\right)$
           \State\hspace*{\algorithmicindent}\hspace*{-1pt}$u_{i}^{k+1} =\prox_{\sigma_i h_i^{*}}\left(u_{i}^{k}+\sigma_i L_{ii}(2x_i^{k+1}-x_i^{k})\right)$           
           \Statex\hspace*{-7pt} {\it\small\% Broadcasting to neighbors}
           \State \hspace*{\algorithmicindent}\hspace*{-1pt}send $x_i^{k+1}$ to all $j\in \mathcal{N}_i^{\rm out}$\quad{\footnotesize(possibly with different delays)}%
         \end{algorithmic}
       \end{algorithm}  
   
   
   
       As shown in \Cref{tm:main-diagonal:RandAsyn}, for small enough stepsizes the generated sequence converges to a primal-dual solution under the bounded delay assumption, and provided that functions $g_i$ are strongly convex. Such needed requirements are summarized below:
   \begin{ass} \label{ass:DiagL:RandAsyn}(stepsize condition) For $i=1,\ldots,m$, the stepsizes $\sigma_i,\gamma_i>0$ satisfy the following assumption:
            \begin{equation}\label{eq:conv-diag:RandAsyn}
             \gamma_{i}<\frac{1}{\sigma_{i}\|L_{ii}\|^{2}+\beta+\tfrac{B^{2}}{2}\|\bar{\beta}\|_{M_{g}^{-1}}^{2}}, \;\, \textrm{where} \quad M_g =   \blkdiag\left(\mu_g^1 I_{n_1},\ldots,\mu_g^mI_{n_m}\right).
           \end{equation}
   \end{ass}
   According to \Cref{ass:DiagL:RandAsyn} a one time global communication of $\|\bar{\beta}\|_{M_{g}^{-1}}$ and $\beta$ is required when initiating the algorithm. 
   
   According to \eqref{eq:conv-diag:RandAsyn} as the  upper bound on delay, $B$, and coupling constants $\bar{\beta}_i$ (as defined in \eqref{eq:Lipz:RandAsyn}) increase, 
    smaller stepsizes should be used. This is intuitive given that in either case the agents have a lower confidence in the currently stored vectors and thus should take smaller steps. Moreover, the higher the modulus of strong convexity, the larger steps agents are allowed to take, countering the effect of the delay.
   
   In the case when the smooth term is separable, the problem is decoupled ($\bar{\beta}_i=0$ for all $i$) and the stepsize condition for each agent does not depend on the delay. The same stepsize condition would be required in the case of synchronous updates ($B=0$). Note that, in this case \eqref{eq:conv-diag:RandAsyn} is still more conservative than the classical results which would require $\gamma_i<1/(\sigma_{i}\|L_{ii}\|^{2}+\beta/2)$ (the difference being a $\beta$ appearing in place of $\beta/2$). 

   Before proceeding with the convergence results, let us define the following 
    \begin{equation}\label{eq:P:RandAsyn}
     P\coloneqq\begin{pmatrix}
       \Gamma^{-1} & -L^\top\\ -L & \Sigma^{-1} 
     \end{pmatrix}.
   \end{equation} 
   Noting that $\Sigma,\Gamma$ are positive definite, and using Schur complement we have that $P$ is positive definite if and only if $\Gamma^{-1}-L^\top \Sigma L$ is positive definite, a condition that holds if \eqref{eq:conv-diag:RandAsyn} is satisfied (since $L$ has a block-diagonal structure). 
   
      In order to establish convergence we first derive the following intermediate result. 
   \begin{lem}\label{lem:gnhineq:RandAsyn}
     Suppose that \Cref{ass:delays:RandAsyn}, \ref{ass:optProb:RandAsyn} and \ref{ass:DiagL-1:RandAsyn}  are satisfied. Consider the sequence generated by \Cref{Alg:PA-VU-Diagonal:RandAsyn}. Then, for any $(x^\star,u^\star)\in\mathcal{S}$ the following hold:
     \begin{enumerate}
       \item \label{eq:fgfirst:RandAsyn}\(\hphantom{\leq{}}\|x^{k+1}-x^{\star}\|_{2M_{g}+\Gamma^{-1}}^{2} -\|x^{k}-x^{\star}\|_{\Gamma^{-1}}^{2}+\|x^{k}-x^{k+1}\|_{\Gamma^{-1}-\beta I}^{2}\)\\[5pt] \noindent
     \(
     \leq 2\sum_{\smash{i=1}}^m\langle \nabla_{i}f(x^{k}[i])-\nabla_i f(x^{k}),x_{i}^{\star}-x_{i}^{k+1}\rangle {}+{} 2\langle L^{\top}(u^{k}-u^{\star}),x^{\star}-x^{k+1}\rangle
   \);   
      \item \label{eq:hfirst:RandAsyn}\(\hphantom{\leq{}}\|u^{k+1}-u^{\star}\|_{\Sigma^{-1}}^{2}-\|u^{k}-u^{\star}\|_{\Sigma^{-1}}^{2}+\|u^{k}-u^{k+1}\|_{\Sigma^{-1}}^{2}\)\\[5pt] \noindent
     \(\leq{}
     2\langle L(2x^{k+1}-x^{k}-x^\star),u^{k+1}-u^{\star}\rangle. 
   \)   
     \end{enumerate}
   \end{lem} 
   \begin{proof}
   To derive the first inequality use \eqref{eq:strconvexity:RandAsyn} with $q=g_i$, $r=x_i^\star$, $\omega_\rho= x_i^{k+1}$ and $\omega=x_{i}^{k}-\gamma_i L_{ii}^{\top}u^{k}[i]-\gamma_i\nabla_{i}f(x^{k}[i])$ (using the update for the primal variable $x_i^{k+1}$ in the algorithm): 
   \begin{align}
     g_{i}(x_{i}^{\star})-g_{i}(x_{i}^{k+1})\geq{}&{}\langle\nabla_{i}f(x^{k}[i])+L_{ii}^{\top}u^{k}_i,x_{i}^{k+1}-x_{i}^{\star}\rangle\nonumber \\&+\tfrac{1}{\gamma_{i}}\langle x_{i}^{k}-x_{i}^{k+1},x_{i}^{\star}-x_{i}^{k+1}\rangle+\tfrac{\mu_{g}^{i}}{2}\|x_{i}^{\star}-x_{i}^{k+1}\|^{2}.\label{eq:g_iLB:RandAsyn}
   \end{align}
   Let $f^\star\coloneqq f(x^\star)$, and $f^k\coloneqq f(x^k)$. By convexity of $f$ we have 
   $  f^k -f^\star  \leq \langle \nabla f(x^k),x^k-x^{\star}\rangle$, 
   and by Lipschitz continuity of $\nabla f$ we have 
    $f^{k+1}\leq f^k+ \langle \nabla f(x^k),x^{k+1}-x^k\rangle + \tfrac{\beta}{2}\|x^{k+1}-x^k\|^2$.
   Summing the last two inequalities yields
   \begin{align}
     f^{k+1}\!-\!f^\star\leq \!\langle \nabla f(x^{k}),x^{k+1}-x^\star\rangle +\tfrac{\beta}{2}\|x^{k+1}-x^{k}\|^{2}.\label{eq:f-threepoints:RandAsyn}
   \end{align}
   For notational convenience we use $F\coloneqq f+g$ and $F^\star\coloneqq F(x^\star)$, $F^k\coloneqq F(x^{k})$. Noting that $g$ is separable, sum \eqref{eq:g_iLB:RandAsyn} over $i$, add \eqref{eq:f-threepoints:RandAsyn} and use \eqref{eq:basic-equality:RandAsyn} (with $V=\Gamma^{-1}$, $a=x^k$, $b=x^{k+1}$, $c=x^\star$):
    \begingroup
   \allowdisplaybreaks
   \begin{align}
     F^{\star}-F^{k+1}\geq{}&\hphantom{{}-{}}\tfrac{1}{2}\|x^{k}-x^{k+1}\|_{\Gamma^{-1}-\beta I}^{2}+\tfrac{1}{2}\|x^{k+1}-x^{\star}\|_{\Gamma^{-1}+M_g}^{2}\nonumber \\&-\tfrac{1}{2}\|x^{k}-x^{\star}\|_{\Gamma^{-1}}^{2}+\sum_{i=1}^{m}\langle-L_{ii}^{\top}u^{k}_i,x_{i}^{\star}-x_{i}^{k+1}\rangle\nonumber \\&+\sum_{i=1}^{\smash{m}}\langle\nabla_{i}f(x^{k})-\nabla_{i}f(x^{k}[i]),x_{i}^{\star}-x_{i}^{k+1}\rangle. \nonumber
   \end{align}
   \endgroup
   On the other hand by convexity of $f$, strong convexity of $g_i$ and \eqref{eq:primal-dual:RandAsyn} we have 
   \begin{equation*}
     F^{k+1}-F^{\star}\geq\langle-L^{\top}u^{\star},x^{k+1}-x^{\star}\rangle+\tfrac{1}{2}\|x^{k+1}-x^{\star}\|_{M_g}^{2}. 
   \end{equation*}
     Summing the last two inequalities, multiplying by 2 and a simple rearrangement yields \eqref{eq:fgfirst:RandAsyn}. 
   
     For the second inequality, consider the update for $u_i^{k+1}$ and use \eqref{eq:strconvexity:RandAsyn}:
   \begin{align}h_{i}^{*}(u_{i}^{\star})-h_{i}^*(u_{i}^{k+1}){}\geq {}&  \langle L_{ii}\left(2x^{k+1}_i-x^{k}_i\right),u_{i}^{\star}-u_{i}^{k+1}\rangle+\tfrac{1}{\sigma_{i}}\langle u_{i}^{k}-u_{i}^{k+1},u_{i}^{\star}-u_{i}^{k+1}\rangle.
   \label{eq:strh:RandAsyn}
   \end{align}
   Furthermore, by convexity of $h_i$ and using \eqref{eq:primal-dual:RandAsyn} we have 
     $h^{*}(u^{k+1})-h^{*}(u^{\star})\geq\langle Lx^{\star},u^{k+1}-u^{\star}\rangle$.
   Sum \eqref{eq:strh:RandAsyn} over all $i$, add the last inequality, and use \eqref{eq:basic-equality:RandAsyn} to derive the inequality.
     \end{proof}
    Our analysis in \Cref{tm:main-diagonal:RandAsyn} relies on showing that the generated sequence is quasi-Fej\'er monotone  with respect to the set of primal-dual solutions in the space  equipped with the inner product $\langle\cdot,\cdot\rangle_P$. Notice that without communication delays ($B\equiv0$), this analysis leads to the usual  Fej\'er monotonicity of the sequence. 
    The use of outdated information introduces additional error terms that are shown to be tolerated by the algorithm if the stepsizes are small enough and the functions $g_i$ are strongly convex. 
   
     \begin{thm}\label{tm:main-diagonal:RandAsyn}
       Suppose that \Cref{ass:optProb:RandAsyn,ass:delays:RandAsyn,ass:DiagL:RandAsyn} are satisfied. 
   Then, the sequence $\seq{z^k}=\seq{x^k,u^k}$ generated by \Cref{Alg:PA-VU-Diagonal:RandAsyn} is $P$-quasi-Fej\'er monotone with respect to $\mathcal{S}$. Furthermore, $\seq{z^k}$ converges to some $z^\star\in\mathcal{S}$. 
     \end{thm}
     \begin{proof}
    
   Adding \cref{eq:fgfirst:RandAsyn} and \cref{eq:hfirst:RandAsyn} we obtain 
   \begin{align} & \|x^{k+1}-x^{\star}\|_{2M_{g}+\Gamma^{-1}}^{2}-\|x^{k}-x^{\star}\|_{\Gamma^{-1}}^{2}+\|x^{k}-x^{k+1}\|_{\Gamma^{-1}}^{2}\nonumber\\
    & +\|u^{k+1}-u^{\star}\|_{\Sigma^{-1}}^{2}-\|u^{k}-u^{\star}\|_{\Sigma^{-1}}^{2}+\|u^{k}-u^{k+1}\|_{\Sigma^{-1}}^{2}\nonumber\\
   \leq {}&\hphantom{{}+{}} 2\sum_{i=1}^{m}\langle x_{i}^{\star}-x_{i}^{k+1},\nabla_{i}f(x^{k}[i])-\nabla_{i}f(x^{k})\rangle + \beta \|x^{k}-x^{k+1}\|^2\nonumber \\&{} +2\langle L(2x^{k+1}-x^{k}-x^\star),u^{k+1}-u^{\star}\rangle +2\langle L^{\top}(u^{k}-u^{\star}),x^{\star}-x^{k+1}\rangle.\label{eq:summedup:RandAsyn}
   \end{align}
   The last two inner products can be rearranged as 
   \begin{align*}
     2\langle L(x^{k+1}-x^{k}),u^{k+1}-u^{k}\rangle-2\langle L(x^{k}-x^{\star}),u^{k}-u^{\star}\rangle  
     +2\langle L(x^{k+1}-x^{\star}),u^{k+1}-u^{\star}\rangle. 
   \end{align*}
   Replacing this term and using \Cref{lem:delayineq-1:RandAsyn} (with $\epsilon_1=2$ and $v=x^{k+1}$) in \eqref{eq:summedup:RandAsyn} yields (with $P$ defined in \eqref{eq:P:RandAsyn}):
     \begin{align}
       &\|z^{k+1}-z^{\star}\|_{P}^{2}-\|z^{k}-z^{\star}\|_{P}^{2}+\|z^{k+1}-z^{k}\|_{P}^{2} \leq  \beta\|x^{k}-x^{k+1}\|^{2} +
       \tfrac{B}{2}\|\bar{\beta}\|_{M_{g}^{-1}}^{2}\sw{1}{1}{x}. \label{eq:Fej-diag:RandAsyn}
     \end{align}
   Sum inequality \eqref{eq:Fej-diag:RandAsyn} over $k$ from $0$ to $p>0$ to obtain:
   \begin{align}
       &\|z^{p+1}-z^{\star}\|_{P}^{2}-\|z^{0}-z^{\star}\|_{P}^{2}+\sum_{k=0}^p\|z^{k+1}-z^{k}\|_{P}^{2} \leq  \beta\sum_{k=0}^p\|x^{k}-x^{k+1}\|^{2} +
       \tfrac{B}{2}\|\bar{\beta}\|_{M_{g}^{-1}}^{2}\sum_{k=0}^p\sw{1}{1}{x}.
       \label{eq:100:RandAsyn} 
     \end{align}
     Let us define 
        \begin{equation*}
     \tilde{P}=\begin{pmatrix}
       \Gamma^{-1}-\tfrac{B^{2}}{2}\|\bar{\beta}\|_{M_{g}^{-1}}^{2}-\beta & -L^\top\\ -L & \Sigma^{-1}
     \end{pmatrix}.
   \end{equation*}
   Since $\Sigma$ is positive definite ($\sigma_i>0$), by Schur complement $\tilde{P}$ is positive definite provided that \eqref{eq:conv-diag:RandAsyn} holds (recall that $L$ has a block-diagonal structure). 
   
     Use \eqref{eq:zeroton:RandAsyn} (with $l=d=1$) in \eqref{eq:100:RandAsyn} to derive 
       $\|z^{p+1}-z^{\star}\|_{P}^{2}+\sum_{k=0}^p\|z^{k+1}-z^{k}\|_{\tilde{P}}^{2}\leq \|z^{0}-z^{\star}\|_{P}^{2}$. 
   Therefore, by letting $p$ to infinity we obtain $\sum_{k=0}^\infty\|z^{k+1}-z^{k}\|_{\tilde{P}}^{2} <\infty$. 
   Hence, using \eqref{eq:zeroton:RandAsyn} for the right-hand side in \eqref{eq:Fej-diag:RandAsyn} we have
   \begin{equation*}
    \sum_{k=0}^\infty  \left(\tfrac{B}{2}\|\bar{\beta}\|_{M_{g}^{-1}}^{2}\sw{1}{1}{x} +\beta\|x^{k}-x^{k+1}\|^{2}\right)<\infty.
    \end{equation*} 
   In view of \eqref{eq:Fej-diag:RandAsyn} we conclude that $\seq{z^k}$ is $P$-quasi-Fej\'er monotone with respect to $\mathcal{S}$. 
   
   Consequently, the sequence $\seq{x^k,u^k}$ is bounded \cite[Lem. 3.1]{Combettes2001QuasiFejerian}. Let $({x}^c,{u}^c)$ be a cluster point of $\seq{x^k,u^k}$, \ie, $(x^{k_n},u^{k_n})\rightarrow ({x}^c,{u}^c)$. Using \Cref{lem:delays:RandAsyn} also $x^{k_n}[i]\rightarrow {x}^c$. 
   Noting that the proximal and linear maps as well as $\nabla f$ are continuous, for all $i=1,\ldots,m$ we have 
   \begin{align*}
     x_i^c =& \prox_{\gamma_i g_i} \big( x_i^c -\gamma_i \nabla_i f(x^c) - \gamma_i L_{ii}^\top u_i^c \big)\\
     u_{i}^{c} =&\prox_{\sigma_i h_i^{\star}}\left(u_{i}^{c}+\sigma_i L_{ii}x_i^{c}\right), 
   \end{align*}
   which implies $(x^c,u^c)\in\mathcal{S}$. The convergence of the sequence follows \cite[Thm. 3.8]{Combettes2001QuasiFejerian}. 
     \end{proof}
   In the case of total coupling (when $L$ is not block-diagonal), it is no longer possible to establish quasi-Fej\'er monotonicity of the V\~u-Condat generated sequence in the space equipped with $\langle\cdot,\cdot\rangle_P$. This is because the coupling linear mapping $L$, is operating on outdated vectors. 
   In the next section we propose an AHU-type primal-dual algorithm that is better suited for problems with total coupling. 
\section{The case of total coupling}
\label{sec:generalL:RandAsyn}

   
   In this section we consider problem \eqref{eq:main-prob:RandAsyn} with \emph{total coupling}. That is, we assume that the coupling between agents is enforced through the linear maps ($L$ is not block-diagonal), and possibly through the smooth term $f$. 
   
   \subsection{An AHU-type primal-dual algorithm}
   We consider the primal-dual algorithm \eqref{Alg:AHU:RandAsyn}. Compared to \eqref{eq:VuCondat:RandAsyn}, in the dual update the linear map $L_{i \mydot}$ operates on $x^k[i]$ in place of $2x^{k+1}[i]-x^k[i]$. This modification results in the possibility of using larger stepsizes since the terms $2x^{k+1}[i]-x^k[i]$ would introduce additional sources of error.   
   
    Let us define the following two sets of indices:
   \begin{align*}
     \mathcal{M}^{\rm p}_i \coloneqq \{j\mid L_{ji} \neq 0\},\quad \mathcal{M}^{\rm d}_i \coloneqq \{j\mid L_{ij} \neq 0\}, 
   \end{align*}
   where $0$ denotes a zero matrix of appropriate dimensions. 
   In \Cref{Alg:PA-VU:RandAsyn}, due to the additional coupling through the linear maps, the primal
   vector of agent $i$ must be transmitted to all $j\in\mathcal{M}^{\rm p}_i\cup \mathcal{N}^{\rm out}_i$ while the dual vector is to be transmitted to all $j\in\mathcal{M}^{\rm d}_i$. Notice that the outdated primal and dual vectors $x^k[i]$ and $u^k[i]$, need not have the same delay pattern and are arbitrary as long as \Cref{ass:delays:RandAsyn} is satisfied, \ie, agent $i$ may use the primal vector $x_j^{k_1}$ and the dual vector $u_j^{k_2}$ that were the variables of agent $j$ at times $k_1$ and $k_2$.
   
    \begin{algorithm}[H]
         \caption{An AHU-type primal-dual algorithm with bounded delays}
         \label{Alg:PA-VU:RandAsyn}
         \begin{algorithmic}[1]
         \algnotext{EndFor}
           \item[\textbf{Initialize:}] $x_i^0\in\R^{n_i}$, $u_i^0\in\R^{r_i}$ for  $i\in\{1,\ldots,m\}$.
           \item[\bf For] $k=0,1,\ldots$ {\bf do}
             \item[\hspace*{13pt}\bf For] {each agent $i=1,\ldots,m$} {\bf do}
                     \Statex\hspace*{-7pt} {\it\small\% Local updates} 
   
           \Statex \hspace*{11pt} {\small\sl\color{black!80}perform the local updates using the last received  information, \ie, the locally stored \linebreak
           \hspace*{12pt} vectors $x^k[i]$ and $u^k[i]$ as  defined in \eqref{eq:outdated:RandAsyn}}:  
           \State\hspace*{\algorithmicindent}\hspace*{-1pt}$x_{i}^{k+1} =\prox_{\gamma_i g_{i}}\left(x_{i}^{k}-\gamma_i L_{\mydot i}^{\top}u^{k}[i]-\gamma_i\nabla_{\!i}f(x^{k}[i])\right)$
           \State\hspace*{\algorithmicindent}\hspace*{-1pt}$u_{i}^{k+1} =\prox_{\sigma_i h_i^{\star}}\!\left(u_{i}^{k}+\sigma_i L_{i\mydot}x^{k}[i]\right)$           
           \Statex\hspace*{-7pt} {\it\small\% Broadcasting to neighbors}
           \State \hspace*{\algorithmicindent}\hspace*{-1pt}send $x_i^{k+1}$ to all $j\in \mathcal{N}_i^{\rm out}\cup \mathcal{M}_i^{\rm p}$, and $u_i^{k+1}$ to all $j\in\mathcal{M}_i^{\rm d}$\quad{\footnotesize(possibly with different delays)}%
         \end{algorithmic}
       \end{algorithm}  
   The convergence of \Cref{Alg:PA-VU:RandAsyn} can be established under the assumption that the functions $g_i$, $h_i^*$ are strongly convex, and provided that small enough stepsizes are used. 
   It is important to note that the strong convexity assumption for  $g$ and $h^*$ is  required even for the synchronous algorithm (refer to the discussion after \eqref{Alg:AHU:RandAsyn}), and is not a restriction that is imposed because of the delays. Moreover, under this assumption the set of primal-dual solutions is a singleton, $\mathcal{S}=\{z^\star\}$.   We summarize the additional requirements below:
   \begin{ass} \label{ass:GenL:RandAsyn}
     For all $i=1,\ldots,m$:
   
     \begin{enumerate}
       \item \label{ass:GenL-2:RandAsyn} (Lipschitz continuity)   
         $h_i$ is continuously differentiable, and $\nabla h_i$ is $\tfrac{1}{\mu_h^i}$-Lipschitz continuous for some $\mu_h^i>0$. Equivalently, $h_i^*$ is $\mu_h^i$-strongly convex;
       \item \label{ass:GenL-3:RandAsyn} (stepsize condition) The stepsizes $\sigma_i,\gamma_i>0$ satisfy the following inequalities 
       \begin{equation*}
         \sigma_i< \frac{1}{C_{\textrm{s}}(B+1)^2},\quad \gamma_i< \frac{1}{\beta+\tfrac{1}{2}R_{\textrm{s}}(B+1)^2+B^2\|\bar{\beta}\|_{M_{g}^{-1}}^{2}},  
       \end{equation*}
       where 
       \begin{equation}\label{eq:RsCs:RandAsyn}
          R_{\textrm s} \coloneqq \sum_{i=1}^m\tfrac{1}{\mu_{h}^{i}}\|L_{i\mydot}\|^{2},\quad C_{\textrm s} \coloneqq \sum_{i=1}^m\tfrac{1}{\mu_{g}^{i}}\|L_{\mydot i}^\top\|^{2}.
        \end{equation} 
     \end{enumerate}
   \end{ass} 
   Note that according to \Cref{ass:GenL-3:RandAsyn} a one time global communication of $R_{\textrm s}$, $C_{\textrm s}$, $\beta$ and $\|\bar{\beta}\|_{M_{g}^{-1}}$ is required. Before proceeding with the convergence results, we define the  following positive definite matrices. 
     \begin{align}
        M_h \coloneqq \blkdiag(\mu_h^1I_{r_1},\ldots,\mu_h^mI_{r_m}), 
        \quad M \coloneqq \blkdiag(M_g,M_h). \label{eq:M:RandAsyn} 
   \end{align}

   The stepsize condition in \Cref{ass:GenL-3:RandAsyn} is more stringent than the condition in \Cref{sec:BlockL:RandAsyn} for the case of partial coupling. This is due to the fact that $L_{i\mydot}$ and $L_{\mydot i}$ are operating on delayed vectors. 
   In the next lemma two key inequalities are established for \Cref{Alg:PA-VU:RandAsyn} that are crucial for our convergence analysis. The proof of the lemma is similar to that of \Cref{lem:gnhineq:RandAsyn} and is therefore omitted. 
   \begin{lem}\label{lem:gnhineq-gen:RandAsyn}
     Suppose that \Cref{ass:delays:RandAsyn}, \ref{ass:optProb:RandAsyn} and \ref{ass:GenL-2:RandAsyn} are satisfied.  Consider the sequence generated by \Cref{Alg:PA-VU:RandAsyn}.  
      Then, for any $(x^\star,u^\star)\in\mathcal{S}$ the following hold:
          \begin{enumerate}
            \item \label{eq:fgfirst-gen:RandAsyn}\(\hphantom{\leq{}}\|x^{k+1}-x^{\star}\|_{2M_{g}+\Gamma^{-1}}^{2} -\|x^{k}-x^{\star}\|_{\Gamma^{-1}}^{2}+\|x^{k}-x^{k+1}\|_{\Gamma^{-1}-\beta I}^{2}\)\\[5pt] \noindent
     \( \leq 
     2\sum_{i=1}^m\langle \nabla_{i}f(x^{k}[i])-\nabla_i f(x^{k})+ L_{\mydot i}^{\top}(u^{k}[i]-u^{\star}),x_{i}^{\star}-x_{i}^{k+1}\rangle
   \);   
      \item \label{eq:hfirst-gen:RandAsyn}\(\hphantom{\leq{}}\|u^{k+1}-u^{\star}\|_{2M_{h}+\Sigma^{-1}}^{2}-\|u^{k}-u^{\star}\|_{\Sigma^{-1}}^{2}+\|u^{k}-u^{k+1}\|_{\Sigma^{-1}}^{2} \)\\[5pt] \noindent
     \(\leq{}
     2\sum_{i=1}^m\langle L_{i\mydot}x^{k}[i],u_{i}^{k+1}-u_{i}^{\star}\rangle    +2\langle Lx^{\star},u^{\star}-u^{k+1}\rangle. 
   \)   
          \end{enumerate}
   \end{lem} 
   It is shown in the next theorem that the sequence generated by \Cref{Alg:PA-VU:RandAsyn} is quasi-Fej\'er monotone in the space equipped with $\langle\cdot,\cdot\rangle_D$ (with $D$ as in \eqref{eq:D:RandAsyn}). 

     \begin{thm}\label{tm:main-gen:RandAsyn}
       Suppose that  \Cref{ass:optProb:RandAsyn,ass:delays:RandAsyn,ass:GenL:RandAsyn} are satisfied. 
   Then, the sequence $\seq{z^k}=\seq{x^k,u^k}$ generated by \Cref{Alg:PA-VU:RandAsyn} is $D$-quasi-Fej\'er monotone with respect to $\mathcal{S}=\{z^\star\}$, and converges to $z^\star$.
     \end{thm}  
     \begin{proof}
       Add \cref{eq:fgfirst-gen:RandAsyn} and \cref{eq:hfirst-gen:RandAsyn}, and rearrange the inner products using \eqref{eq:TheObvs:RandAsyn} to derive (with $D$ defined in \eqref{eq:D:RandAsyn})
   \begin{align} & \|z^{k+1}-z^{\star}\|_{D}^{2}-\|z^{k}-z^{\star}\|_{D}^{2}+\|z^{k}-z^{k+1}\|_{D}^{2}\nonumber\\
   \leq{} & \hphantom{{}+{}}2\sum_{i=1}^{m}\langle \nabla_{i}f(x^{k}[i])-\nabla_{i}f(x^{k}) + L_{\mydot i}^{\top}(u^{k}[i]-u^{k+1}),x_{i}^{\star}-x_{i}^{k+1}\rangle -\|u^{k+1}-u^{\star}\|_{2M_{h}}^{2} \nonumber\\ & + 2\sum_{i=1}^{m}\langle L_{i\mydot}(x^{k}[i]-x^{k+1}),u_{i}^{k+1}-u_{i}^{\star}\rangle + \beta\|x^{k}-x^{k+1}\|^{2} - \|x^{k+1}-x^{\star}\|_{2M_{g}}^{2}
   \label{eq:uglyineq:RandAsyn}
   \end{align}
   
   Using the inequalities in \Cref{lem:delayineq:RandAsyn} with $q=1$, $\epsilon_1=\epsilon_2=1$, $\epsilon_3=2$, $v=x^{k+1}$ and $y=u^{k+1}$:
    \begin{align} 
       \|z^{k+1}-z^{\star}\|_{D}^{2}-\|z^{k}-z^{\star}\|_{D}^{2}+\|z^{k}-z^{k+1}\|_{D}^{2} 
   \leq{}&{} B\|\bar{\beta}\|_{M_{g}^{-1}}^{2}\sw{1}{1}{x} +\tfrac{1}{2}R_{\textrm{s}}(B+1)\sw{1}{0}{x} \nonumber\\
    & {}+C_{\textrm{s}}(B+1)\sw{1}{0}{u}+\beta\|x^{k}-x^{k+1}\|^{2}. \label{eq:lemmain:RandAsyn} 
   \end{align}
    Sum over $k$ from $0$ to $p>0$ to obtain
   {\mathtight\begin{align}  \|z^{p+1}-z^{\star}\|_{D}^{2}-\|z^{0}-z^{\star}\|_{D}^{2}+\sum_{k=0}^{p}\|z^{k}-z^{k+1}\|_{D}^{2}
   \leq {}
    & B\|\bar{\beta}\|_{M_{g}^{-1}}^{2}\sum_{k=0}^{p}S_{1}^{1}(x^{t})_{t\leq k}+\tfrac{R_{\textrm{s}}(B+1)}{2}\sum_{k=0}^{p}S_{1}^{0}(x^{t})_{t\leq k}\nonumber\\ & + \beta\sum_{k=0}^{p}\|x^{k}-x^{k+1}\|^{2}+{C_{\textrm{s}}(B+1)}\sum_{k=0}^{p}S_{1}^{0}(u^{t})_{t\leq k}.
    \label{eq:23:RandAsyn} 
   \end{align}}
   By repeated use of \eqref{eq:zeroton:RandAsyn} in \eqref{eq:23:RandAsyn} we obtain
   {\mathtight\begin{align*} \|z^{p+1}-z^{\star}\|_{D}^{2}-\|z^{0}-z^{\star}\|_{D}^{2}+\sum_{k=0}^{p}\|z^{k}-z^{k+1}\|_{D}^{2}
   \leq {}&{}  \left(B^{2}\|\bar{\beta}\|_{M_{g}^{-1}}^{2}+\tfrac{1}{2}R_{\textrm{s}}(B+1)^2+\beta\right)\sum_{k=0}^{p}\|x^{k}-x^{k+1}\|^{2} \\& + C_{\textrm{s}}(B+1)^{2}\sum_{k=0}^{p}\|u^{k}-u^{k+1}\|^{2}.
   \end{align*}}
   If the stepsizes are small enough to satisfy \Cref{ass:GenL-3:RandAsyn}, letting  $p$ to infinity yields  
     $\sum_{k=0}^{\infty}\|z^{k+1}-z^{k}\|^{2}<\infty$. 
   Therefore, it follows from \eqref{eq:lemmain:RandAsyn} (using \eqref{eq:zeroton:RandAsyn}) that $\seq{z^k}$ is $D$-quasi-Fej\'er monotone with respect to $\mathcal{S}$. Arguing as in \Cref{tm:main-diagonal:RandAsyn} completes the proof.   
     \end{proof}
   The next theorem provides a sufficient condition for the stepsizes under which linear convergence is attained.  
     \begin{thm}[linear convergence]\label{tm:main-gen-linear:RandAsyn}
       Suppose that \Cref{ass:delays:RandAsyn}, \ref{ass:optProb:RandAsyn} and \ref{ass:GenL-2:RandAsyn} are satisfied. Consider the sequence $\seq{z^k}$ generated by \Cref{Alg:PA-VU:RandAsyn}.  
       Let $c$ be a positive scalar and set $\gamma_i=\tfrac{c}{\mu_g^i}, \sigma_i=\tfrac{c}{\mu_h^i}$ for $i=1,\ldots,m$. Let $\mu_g^{\min}=\min\{\mu_g^1,\ldots,\mu_g^m\}$, $\mu_h^{\min}=\min\{\mu_h^1,\ldots,\mu_h^m\}$. Then, the following linear convergence rate holds
       \begin{equation*}
        \|z^k-z^\star\|^2_D \leq  \big(\tfrac{1}{1+c}\big)^k\|z^0-z^\star\|^2_D,
       \end{equation*}
       provided that $c\leq(1+c_{2})^{\tfrac{1}{B+1}}-1$ where 
   $c_{2}=\min\set{\frac{\mu_{g}^{\min}}{2B\|\bar{\beta}\|_{M_{g}^{-1}}^{2}+R_{\textrm{s}}(B+1)+\beta},\frac{\mu_{h}^{\min}}{2C_{\textrm{s}}(B+1)}}$.   
         \end{thm}    
         \begin{proof}
           In \Cref{tm:main-gen:RandAsyn} the strong convexity assumption was leveraged to counteract the error terms. In order to prove linear convergence we retain some of the strong convexity terms. 
   Using the inequalities of \Cref{lem:delayineq:RandAsyn} with $q=1$, $\epsilon_1=\epsilon_2=0.5$, $\epsilon_3=1$, $v=x^{k+1}$ and $y=u^{k+1}$  in \eqref{eq:uglyineq:RandAsyn} yields 
   \begin{align}
     &\|z^{k+1}-z^{\star}\|_{D}^{2}-\|z^{k}-z^{\star}\|_{D}^{2}+\|z^{k}-z^{k+1}\|_{D}^{2} +\|x^{k+1}-x^{\star}\|_{M_{g}}^{2}+\|u^{k+1}-u^{\star}\|_{M_{h}}^{2}\nonumber\\
   \leq {} &  \left(2B\|\bar{\beta}\|_{M_{g}^{-1}}^{2}+R_{\textrm{s}}(B+1)\right)S_{1}^{0}(x^{t})_{t\leq k}+2C_{\textrm{s}}(B+1)S_{1}^{0}(u^{t})_{t\leq k}+\beta\|x^{k}-x^{k+1}\|^{2}.\label{eq:ugly2:RandAsyn}
   \end{align}
   Note that one may set these constants differently and obtain a different valid bound on the stepsizes.   
   
   Since we set $\gamma_i=\tfrac{c}{\mu_g^i}$, $\sigma_i=\tfrac{c}{\mu_h^i}$, we have 
   $
   D= \blkdiag(\Gamma^{-1},\Sigma^{-1})= \tfrac{1}{c}\blkdiag(M_g,M_h),
   $
   which together with \eqref{eq:ugly2:RandAsyn} yields 
    \begin{align}
    (1+c)\|z^{k+1}-z^{\star}\|_{D}^{2} - \|z^{k}-z^{\star}\|_{D}^{2} 
    \leq {}& \left(2B\|\bar{\beta}\|_{M_{g}^{-1}}^{2}+ R_{\textrm{s}}(B+1)+\beta\right)\sw{1}{0}{x}\nonumber\\&+2C_{\textrm{s}}(B+1)\sw{1}{0}{u}-\|z^{k+1}-z^{k}\|^{2}_D,\label{eq:linin:RandAsyn}
   \end{align}
   where we used the conservative bound $\beta\|x^{k+1}-x^k\|^2\leq\beta\sw{1}{0}{x}$ in order to avoid algebraic difficulties.
   The result follows by multiplying \eqref{eq:linin:RandAsyn} by $(1+c)^k$ and summing over $k$ from $0$ to $p$, see \cite[Lem. 1]{Aytekin2016Analysis}.
         \end{proof}

\subsection{A randomized variant}
\label{sec:rand:RandAsyn}

   
   In this subsection we propose a randomized variant of \Cref{Alg:PA-VU:RandAsyn} where agents are activated randomly according to independent probabilities, \ie, at every iteration several agents may be active. Unlike the partially asynchronous protocol \cite{Bertsekas1989Parallel}, in this scheme 
    the agents are not required to perform at least one update in any interval of length $B$. 
     \begin{algorithm}
         \caption{A randomized variant of \protect\Cref{Alg:PA-VU:RandAsyn}}
         \label{Alg:PA-VU-Rand:RandAsyn}
         \begin{algorithmic}[1]
         \algnotext{EndFor}
           \item[\textbf{Initialize:}] $x_i^0\in\R^{n_i}$, $u_i^0\in\R^{r_i}$ for  $i\in\{1,\ldots,m\}$.
           \item[\bf For] $k=0,1,\ldots$ {\bf do}
           \Statex \hspace*{-10pt} each agent $i=1,\ldots,m$ is activated independantly with  probability $p_i>0$.
             \item[\hspace*{13pt}\bf For] {active agents} {\bf do}
                     \Statex\hspace*{-7pt} {\it\small\% Local updates}
   
           \Statex \hspace*{11pt} {\small\sl\color{black!80}perform the local updates using the last received  information, \ie, the locally stored\linebreak
           \hspace*{12pt} vectors $x^k[i]$ and $u^k[i]$ as  defined in \eqref{eq:outdated:RandAsyn}}:  
           \State\hspace*{\algorithmicindent}\hspace*{-1pt}$x_{i}^{k+1} =\prox_{\gamma_i g_{i}}\left(x_{i}^{k}-\gamma_i L_{\mydot i}^{\top}u^{k}[i]-\gamma_i\nabla_{\!i}f(x^{k}[i])\right)$
           \State\hspace*{\algorithmicindent}\hspace*{-1pt}$u_{i}^{k+1} =\prox_{\sigma_i h_i^{\star}}\!\left(u_{i}^{k}+\sigma_i L_{i\mydot}x^{k}[i]\right)$           
           \Statex\hspace*{-7pt} {\it\small\% Broadcasting to neighbors}
           \State \hspace*{\algorithmicindent}\hspace*{-1pt}send $x_i^{k+1}$ to all $j\in \mathcal{N}_i^{\rm out}\cup \mathcal{M}_i^{\rm p}$, and $u_i^{k+1}$ to all $j\in\mathcal{M}_i^{\rm d}$\quad{\footnotesize(possibly with different delays)}%
         \end{algorithmic}
       \end{algorithm}  
   In the randomized setting of \Cref{Alg:PA-VU-Rand:RandAsyn}, the stepsize condition in \Cref{ass:GenL-3:RandAsyn} is replaced by the following stepsize condition. 
   \begin{ass}(stepsize condition) \label{ass:GenL-Rand:RandAsyn}
     For all $i=1,\ldots,m$, 
       independent probabilities $p_i>0$ 
       and stepsizes $\sigma_i,\gamma_i>0$ satisfy the following inequalities 
    \begin{equation*}
     \sigma_{i}<\frac{1}{2C_{s}\left(B^{2}p_{i}+1\right)},\quad \gamma_{i}<\frac{1}{\beta+R_{s}\left(B^{2}p_{i}+1\right)+\|\bar{\beta}\|_{M_{g}^{-1}}^{2}B^{2}p_{i}}.
       \end{equation*}
   \end{ass} 
   Notice that compared to the non-randomized version, according to \Cref{ass:GenL-Rand:RandAsyn}, an agent is allowed to take larger steps if its probability of activation is smaller.  
   \begin{thm}\label{thm:rand:RandAsyn}
   	Suppose that \Cref{ass:delays:RandAsyn}, \ref{ass:optProb:RandAsyn}, \ref{ass:GenL-2:RandAsyn}
   and \ref{ass:GenL-Rand:RandAsyn} are satisfied. Then, the sequence $\seq{z^k}=\seq{x^k,u^k}$ generated by \Cref{Alg:PA-VU-Rand:RandAsyn} converges almost surely to $z^\star$. 
   \end{thm}
   \begin{proof}
     Let $\bar{z}^{k+1}_i=(\bar{x}^{k+1}_i,\bar{u}^{k+1}_i)$ denote the updated vector belonging to agent $i$ if that agent was to perform an update at iteration $k$. That is, in \Cref{Alg:PA-VU-Rand:RandAsyn}, $z_i^{k+1}=\bar{z}_i^{k+1}$ if agent $i$ is activated and $z_i^{k+1}={z}_i^{k}$ if it remains idle. 
   Let us define the global vector $\bar{z}^{k+1}=(\bar{z}_1^{k+1},\ldots,\bar{z}_m^{k+1})$ which corresponds to a deterministic update of all agents at iteration $k$. 
   Using \Cref{lem:gnhineq-gen:RandAsyn} as in \eqref{eq:uglyineq:RandAsyn} we have 
    \begingroup
   \allowdisplaybreaks
   \begin{align} & \|\bar{z}^{k+1}-z^{\star}\|_{D}^{2}-\|z^{k}-z^{\star}\|_{D}^{2}+\|z^{k}-\bar{z}^{k+1}\|_{D}^{2}\nonumber\\
    & -\beta\|x^{k}-\bar{x}^{k+1}\|^{2}+\|\bar{x}^{k+1}-x^{\star}\|_{2M_{g}}^{2}+\|\bar{u}^{k+1}-u^{\star}\|_{2M_{h}}^{2}\nonumber\\
   \leq {}& \hphantom{{}+{}} 2\sum_{i=1}^{m}\langle \nabla_{i}f(x^{k}[i])-\nabla_{i}f(x^{k})+L_{\mydot i}^{\top}(u^{k}[i]-\bar{u}^{k+1}), x_{i}^{\star}-\bar{x}_{i}^{k+1}\rangle\nonumber\\
    & +2\sum_{i=1}^{m}\langle L_{i\mydot}(x^{k}[i]-\bar{x}^{k+1}),\bar{u}_{i}^{k+1}-u_{i}^{\star}\rangle
    \label{eq:uglyineq-2:RandAsyn}
   \end{align}
   \endgroup
   Using \Cref{lem:delayineq-3:RandAsyn} with $q=0$, $v=\bar{x}^{k+1}$, and \eqref{eq:simpleYonug:RandAsyn} (with $\epsilon=\epsilon_2\mu_g^i$ for each term in the summation) we have
   {\mathtight\begin{align}
   \sum_{i=1}^{m}  \langle L_{\mydot i}^{\top}(u^{k}[i]-u^{k}+{u}^{k}-\bar{u}^{k+1}),x_{i}^{\star}-\bar{x}_{i}^{k+1}\rangle
   \leq {} & \tfrac{\epsilon_{2}}{2}\|x^{\star}-\bar{x}^{k+1}\|_{M_{g}}^{2}+\tfrac{BC_{s}}{2\epsilon_{2}}\sw{1}{1}{u}\nonumber\\
    & +\sum_{i=1}^{m}\left(\|L_{\mydot i}^{\top}\|^{2}\tfrac{1}{2\epsilon_{2}\mu_{g}^{i}}\|u^{k}-\bar{u}^{k+1}\|^{2}+\tfrac{\epsilon_{2}\mu_{g}^{i}}{2}\|x_{i}^{\star}-\bar{x}_{i}^{k+1}\|^{2}\right)\nonumber\\
   = {} & \epsilon_{2}\|x^{\star}-\bar{x}^{k+1}\|_{M_{g}}^{2}+\tfrac{C_{s}}{2\epsilon_{2}}\|\bar{u}^{k+1}-u^{k}\|^{2}+\tfrac{BC_{s}}{2\epsilon_{2}}\sw{1}{1}{u}.\label{eq:inner_prod_new:RandAsyn}
   \end{align}}
   Similarly, using \Cref{lem:delayineq-4:RandAsyn} with $q=0$, $y=\bar{u}^{k+1}$ and \eqref{eq:simpleYonug:RandAsyn} (with $\epsilon=\epsilon_3\mu_h^i$ for each term in the summation) we obtain: 
    {\mathtight\begin{align}
     \sum_{i=1}^{m}\langle L_{i\mydot}(x^{k}[i]-\bar{x}^{k+1}),\bar{u}_{i}^{k+1}-u_{i}^{\star}\rangle  \leq{}&{} \epsilon_{3}\|\bar{u}^{k+1}-u^{\star}\|_{M_{h}}^{2} +\tfrac{R_{s}}{2\epsilon_{3}}\|x^{k}-\bar{x}^{k+1}\|^{2} +\tfrac{BR_{\textrm{s}}}{2\epsilon_{3}}\sw{1}{1}{x}. \label{eq:inner_prod_new2:RandAsyn}
   \end{align} }
   Using \eqref{eq:inner_prod_new:RandAsyn}, \eqref{eq:inner_prod_new2:RandAsyn} with $\epsilon_2=0.5$ and $\epsilon_3=1$ together with \Cref{lem:delayineq-1:RandAsyn} with $\epsilon_1=1$, $v=\bar{x}^{k+1}$  in \eqref{eq:uglyineq-2:RandAsyn} yields
   \begin{align} & \|\bar{z}^{k+1}-z^{\star}\|_{D}^{2}-\|z^{k}-z^{\star}\|_{D}^{2}+\|z^{k}-\bar{z}^{k+1}\|_{D}^{2}\nonumber\\
   \leq {}& a\sw{1}{1}{x}+b\sw{1}{1}{u}+\left(\beta+R_{s}\right)\|x^{k}-\bar{x}^{k+1}\|^{2} +2C_{s}\|\bar{u}^{k+1}-u^{k}\|^{2},\label{eq:main_ineq:RandAsyn}
   \end{align}
   where $a\coloneqq\left(BR_{s}+B\|\bar{\beta}\|_{M_{g}^{-1}}^{2}\right)$ and $b\coloneqq2BC_{s}$.  
   
   Let $\mathbb{E}_k[\cdot]$ denote the expectation conditioned on the knowledge until time $k$. 
   Moreover, for notational convenience let us define the diagonal probability matrix 
   \begin{align*}
     \Pi=  \blkdiag(p_1 \I_{n_1},\ldots,p_m \I_{n_m},p_1 \I_{r_1},\ldots,p_m \I_{r_m}), \quad D_i =  \blkdiag(\gamma_i^{-1} \I_{n_i}\sigma_i^{-1} \I_{r_i}),
   \end{align*}
   $z_i^k=(x_i^k,u_i^k)$ and $z_i^\star=(x_i^\star,u_i^\star)$. 
   Consequently, using the fact that $D$ is diagonal we have 
    \begingroup
   \allowdisplaybreaks
     \begin{align}
       \mathbb{E}_{k}\left\{ \|z^{k+1}-z^{\star}\|_{\Pi^{-1}D}^{2}\right\} & =\mathbb{E}_{k}\left[\sum_{i=1}^{m}p_{i}^{-1}\|z_{i}^{k+1}-z_{i}^{\star}\|_{D_{i}}^{2}\right]\nonumber \\
     &=\sum_{i=1}^{m}p_{i}^{-1}\left(p_{i}\|\bar{z}_{i}^{k+1}-z_{i}^{\star}\|_{D_{i}}^{2}+(1-p_{i})\|z_{i}^{k}-z_{i}^{\star}\|_{D_{i}}^{2}\right) \nonumber\\ &
     =\sum_{i=1}^{m}\left(\|\bar{z}_{i}^{k+1}-z_{i}^{\star}\|_{D_{i}}^{2}+\tfrac{(1-p_{i})}{p_{i}}\|z_{i}^{k}-z_{i}^{\star}\|_{D_{i}}^{2}\right) \nonumber\\ &
     =\|\bar{z}^{k+1}-z^{\star}\|_{D}^{2}+\|z^{k}-z^{\star}\|_{\Pi^{-1}D}^{2}-\|z^{k}-z^{\star}\|_{D}^{2}. \label{eq:Expe:RandAsyn}
     \end{align}
     \endgroup
   Therefore, using \eqref{eq:main_ineq:RandAsyn} we obtain
   \begin{align}
     \mathbb{E}_{k}\left\{ \|z^{k+1}-z^{\star}\|_{\Pi^{-1}D}^{2}\right\}   \leq{}&\|z^{k}-z^{\star}\|_{\Pi^{-1}D}^{2}-\|z^{k}-\bar{z}^{k+1}\|_{D}^{2}+2C_{s}\|\bar{u}^{k+1}-u^{k}\|^{2} \nonumber\\&
     +a\sw{1}{1}{x}+b\sw{1}{1}{u}+
     \left(\beta+R_{s}\right)\|x^{k}-\bar{x}^{k+1}\|^{2}.\label{eq:Sto-Fej:RandAsyn}
   \end{align}
   Let 
     $X^{k}\coloneqq\sum_{\tau=[k-B]_+}^{k-1}\left(\tau-(k-B)+1\right)\|x^{\tau+1}-x^\tau\|^2$ and $U^{k}\coloneqq\sum_{\tau=[k-B]_+}^{k-1}\left(\tau-(k-B)+1\right)\|u^{\tau+1}-u^\tau\|^2$. 
   It is easy to see that
   \begin{equation*}
    X^{k+1} =X^{k}-\sum_{\tau=[k-B]_+}^{k-1}\|x^{\tau+1}-x^{\tau}\|^{2}+B\|x^{k}-x^{k+1}\|^{2}.
   \end{equation*}
   Arguing as in \eqref{eq:Expe:RandAsyn} we have 
   $\mathbb{E}_{k}\left\{ \|x^{k+1}-x^{k}\|^{2}\right\} 
   =  \sum_{i=1}^m p_i\|\bar{x}^{k+1}_i-x_i^{k}\|^{2}$. 
   Therefore 
   \begin{align}
     \mathbb{E}_{k}\left\{ X^{k+1}\right\} \leq{} X^{k}-\sum_{\tau=[k-B]_+}^{k-1}\|x^{\tau+1}-x^{\tau}\|^{2}+B\sum_{i=1}^m p_i\|\bar{x}^{k+1}_i-x_i^{k}\|^{2},\label{eq:EXk:RandAsyn}
   \end{align}
   Similarly for the dual variables we have 
   \begin{align}
     \mathbb{E}_{k}\left\{ U^{k+1}\right\} \leq{} U^{k}-\sum_{\tau=[k-B]_+}^{k-1}\|u^{\tau+1}-u^{\tau}\|^{2} +B\sum_{i=1}^m p_i\|\bar{u}^{k+1}_i-u_i^{k}\|^{2}.\label{eq:EUk:RandAsyn}
   \end{align}
   Consider the Lyapunov function  $v^k \coloneqq \|z^{k}-z^{\star}\|_{\Pi^{-1}D}^{2} + aX^k +bU^k$. 
   Then, using \eqref{eq:Sto-Fej:RandAsyn}, \eqref{eq:EXk:RandAsyn} and \eqref{eq:EUk:RandAsyn} we obtain 
   \begin{align*}
     \mathbb{E}_{k}\left\{ v^{k+1}\right\}  \leq {}&v^{k}-\|z^{k}-\bar{z}^{k+1}\|_{D}^{2} +aB\sum_{i=1}^{m}p_{i}\|\bar{x}_{i}^{k+1}-x_{i}^{k}\|^{2}+bB\sum_{i=1}^{m}p_{i}\|\bar{u}_{i}^{k+1}-u_{i}^{k}\|^{2} \nonumber \\&
     +\left(\beta+R_{s}\right)\|x^{k}-\bar{x}^{k+1}\|^{2}+2C_{s}\|\bar{u}^{k+1}-u^{k}\|^{2}.
   \end{align*}
   Therefore,  if \Cref{ass:GenL-Rand:RandAsyn} holds then there exists $\bar{c}>0$ such that  
     $\mathbb{E}_{k}\left\{ v^{k+1}\right\}  \leq v^{k}-\bar{c}\|z^{k}-\bar{z}^{k+1}\|^{2}$.
   Since $\|z^{k}-{z}^{k+1}\|\leq \|z^{k}-\bar{z}^{k+1}\|$, we conclude
   by the Robbins-Siegmund lemma \cite{Robbins1985convergence} 
    that almost surely $\|z^{k}-{z}^{k+1}\|$ converges to zero, and consequently by \Cref{lem:delays:RandAsyn} so does $\|z^{k}-{z}^{k}[i]\|$. 
   Moreover, as a second consequence of the Robbins-Siegmund lemma we have that $\seq{v^k}$ and in particular $\|z^k-z^\star\|_{\Pi^{-1}D}$ converges to some $[0,\infty)$-valued variable. The convergence result follows by standard arguments as in \cite[Thm. 3]{Bianchi2016Coordinate} and \cite[Prop. 2.3]{Combettes2015Stochastic} and using continuity of the proximal operator.
   \end{proof}
   In the next theorem we establish linear convergence for \Cref{Alg:PA-VU-Rand:RandAsyn} and provide an explicit convergence rate.  
   \begin{thm}\label{thm:rand-linear:RandAsyn}
       Suppose that \Cref{ass:delays:RandAsyn}, \ref{ass:optProb:RandAsyn} and \ref{ass:GenL-2:RandAsyn} are satisfied. Let $c<\min\{p_1,\ldots,p_m\}$ be a positive scalar and set $\gamma_i=\tfrac{1}{(p_i/c-1)\mu_g^i}$ and $\sigma_i=\tfrac{1}{(p_i/c-1)\mu_h^i}$ for $i=1,\ldots,m$. Moreover, let 
   \begin{align*}
     \delta_1=&\frac{1}{2B\|\bar{\beta}\|_{M_{g}^{-1}}^{2}+2BR_{\textrm{s}}+2R_{s}+\beta}\min_{i}\{(p_{i}-c)\mu_{g}^{i}\}, \nonumber\\
   \delta_2=&\frac{1}{4C_{s}(1+B)}\min_{i}\{(p_{i}-c)\mu_{h}^{i}\}.
   \end{align*}
       Suppose that $c$ is such that it satisfies 
     $\frac{1}{(1-c)^{B}}+c\leq 1+\min\{\delta_1,\delta_2\}$.
   (such $c$ always exist close enough to zero).  
   Then, the following holds for the sequence $\seq{z^k}$ generated by  \Cref{Alg:PA-VU-Rand:RandAsyn}:
   \begin{equation*}
     \mathbb{E}\{\|z^k-z^\star\|^2_{{M}}\} \leq  \left( {1-c} \right)^k\|z^0-z^\star\|^2_{{M}}. 
   \end{equation*}
   \end{thm}
   \begin{proof}
     As in the deterministic case in \Cref{tm:main-gen-linear:RandAsyn}, in order to show linear convergence we  retain some of the strong convexity terms. Consider \eqref{eq:uglyineq-2:RandAsyn} and use \Cref{lem:delayineq-1:RandAsyn} with $\epsilon_1=0.5$, $v=\bar{x}^{k+1}$, \eqref{eq:inner_prod_new:RandAsyn} and \eqref{eq:inner_prod_new2:RandAsyn} with $\epsilon_2=0.25$, $\epsilon_3=0.5$ to derive: 
   \begin{align} & \|\bar{z}^{k+1}-z^{\star}\|_{D+M}^{2}-\|z^{k}-z^{\star}\|_{D}^{2}+\|\bar{z}^{k+1}-z^{k}\|_{D}^{2}\nonumber\\
   \leq {} & 2B\left(\|\bar{\beta}\|_{M_{g}^{-1}}^{2}+R_{\textrm{s}}\right)\sw{1}{1}{x}+4BC_{s}\sw{1}{1}{u}\nonumber\\
    & +(2R_{s}+\beta)\|\bar{x}^{k+1}-x^{k}\|^{2}+4C_{s}\|\bar{u}^{k+1}-u^{k}\|^{2},\label{eq:uglyineq-4:RandAsyn}
   \end{align}
   where  $M\coloneqq \blkdiag(M_{g},M_{h})$. Given the choice of stepsizes we have $D=\blkdiag(\Gamma^{-1},\Sigma^{-1})=\left(\tfrac{1}{c}\Pi-\I\right)M$. Using this and arguing as in \eqref{eq:Expe:RandAsyn}  we have:  
   \begin{align*}\mathbb{E}_{k}  \left\{ \|z^{k+1}-z^{\star}\|_{M}^{2}\right\} 
     ={}&\|\bar{z}^{k+1}-z^{\star}\|_{\Pi M}^{2}+\|z^{k}-z^{\star}\|_{M}^{2}-\|z^{k}-z^{\star}\|_{\Pi M}^{2}\\
     ={}&c\|\bar{z}^{k+1}-z^{\star}\|_{D+M}^{2}+\|z^{k}-z^{\star}\|_{M}^{2}-c\|z^{k}-z^{\star}\|_{D+M}^{2}.
   \end{align*}
   Combining this with \eqref{eq:uglyineq-4:RandAsyn} yields  
   \begin{align}  \mathbb{E}_{k}\left\{ \|z^{k+1}-z^{\star}\|_{M}^{2}\right\} 
   \leq {}& (1-c)\|z^{k}-z^{\star}\|_{M}^{2}-c\|\bar{z}^{k+1}-z^{k}\|_{D}^{2}\nonumber\\
    & +2cB\left(\|\bar{\beta}\|_{M_{g}^{-1}}^{2}+R_{\textrm{s}}\right)\sw{1}{1}{x}+4cBC_{s}\sw{1}{1}{u}\nonumber\\
    & +c(2R_{s}+\beta)\|\bar{x}^{k+1}-x^{k}\|^{2}+4cC_{s}\|\bar{u}^{k+1}-u^{k}\|^{2}.\label{eq:explin:RandAsyn}
   \end{align}
   Next, note that by definition of $\bar{z}^k=(\bar{x}^k,\bar{u}^k)$ we have 
   \begin{align*}
     \sw{1}{1}{x} \leq \sum_{\tau=[k-B]_+}^{\smash{k-1}} \|\bar{x}^{\tau+1}-x^\tau\|^2  \leq  \sum_{\tau=[k-B]_+}^{\smash{k}} \|\bar{x}^{\tau+1}-x^\tau\|^2,
   \end{align*}
   and similarly for the dual vector. Using this in \eqref{eq:explin:RandAsyn} yields:
    \begingroup
   \allowdisplaybreaks
   \begin{align*}\mathbb{E}_{k}  \left\{ \|z^{k+1}-z^{\star}\|_{M}^{2}\right\} 
     \leq{}&(1-c)\|z^{k}-z^{\star}\|_{M}^{2}-c\|\bar{z}^{k+1}-z^{k}\|_{D}^{2}\nonumber\\
     & +c\left(2B\|\bar{\beta}\|_{M_{g}^{-1}}^{2}+2BR_{\textrm{s}}+2R_{s}+\beta\right)\sum_{\tau=[k-B]_{+}}^{k}\|\bar{x}^{\tau+1}-x^{\tau}\|^{2}\nonumber\\
    & +4cC_{s}(1+B)\sum_{\tau=[k-B]_{+}}^{k}\|\bar{u}^{\tau+1}-u^{\tau}\|^{2}.
   \end{align*}
   \endgroup
   The result follows by taking total expectation from both sides,   
     dividing by ${(1-c)^{k+1}}$ and summing over $k$ from $0$ to $p$, see \cite[Lem. 1]{Aytekin2016Analysis}. 
   \end{proof}
   Note that owing to the diagonal metric used in the proofs of \Cref{thm:rand-linear:RandAsyn,thm:rand:RandAsyn}, the independent activation pattern in \Cref{Alg:PA-VU-Rand:RandAsyn} can be replaced with the more general random sweeping strategy as in \cite{Bianchi2016Coordinate,Latafat2019Randomized}. 

   \section{Numerical simulations}\label{sec:Formation}


In this section we revisit the formation control example defined in \eqref{eq:formation}. For the dynamics of each agent/robot we used the model of \cite{Schouwenaars2004Decentralized} with exact discretization of steplength $\Delta T=1$. The state and input cost matrices and the constraints sets $\mathcal{W}_i$ are as in \cite[\S VI]{Latafat2019Randomized}. 

Let $\hat{C}$ be a linear mapping such that $\hat{C}w_i= C\xi_i$ and $L_{i\mydot}$ be such that $L_{i\mydot}w=(E_iw_i,w_i)$. Minimization \eqref{eq:formation} can be formulated as an instance of \eqref{eq:main-prob:RandAsyn} by setting $g_i(w_i)=\tfrac{1}2w_i^\top Q_i w_i$, $f(w)=\sum_{i=1}^m \tfrac{\lambda_i}{2}\sum_{j\in\mathcal{A}_i}\|\hat{C}(w_i- w_j) - d_{ij}\|^2$,  
 $h_{i}(y_{i},v_{i})=\delta_{b_{i}}(y_{i})+\delta_{\mathcal{W}_{i}}(v_{i})$.  Therefore, implementation of the algorithms presented in this paper would only require simple operations such as matrix-vector products, projections onto points and projections onto sets $\mathcal{W}_i$ (which are simple boxes). 

In our simulations, horizon length $3$ was used. The delays between agents are randomly generated integers in the interval $[0,B]$. We consider two numerical simulations. In the first one we set $m=5$, $B=1$ with initial polygon configuration and enforce an arrow formation by appropriate selection of $d_{ij}$. The interested reader may refer to \cite[\S VI]{Latafat2019Randomized} for additional details of the formation setup. As discussed in the introduction,  minimization \eqref{eq:formation} is an example of partial coupling and \Cref{Alg:PA-VU-Diagonal:RandAsyn} is the suitable choice. In the first numerical experiment, depicted in \Cref{fig2} (left), we use the theoretical stepsize bound in \eqref{eq:conv-diag:RandAsyn}. For comparison, we also considered the dual decomposition approach of \cite{Raffard2004Distributed} that is based on the subgradient method (although this algorithm is not studied with communication delays). For comparison, \Cref{Alg:PA-VU:RandAsyn,Alg:PA-VU-Rand:RandAsyn} are also plotted even though they are not designed for this type of problem. \Cref{Alg:PA-VU-Rand:RandAsyn} is used with probabilities of activation $p_i$ set to $0.2$ and $0.8$. It is observed that the convergence rate of the proposed algorithms are linear.  

In the second numerical experiment, depicted in \Cref{fig2} (right), we considered a larger problem with $m=50$ and the maximum delay $B=10$. We simulated the algorithms with nominal stepsizes. It is observed that \Cref{Alg:PA-VU:RandAsyn} and the dual decomposition approach struggle to reach a high precision. Interestingly, the randomized algorithm \Cref{Alg:PA-VU-Rand:RandAsyn} is able to overcome this. Moreover, even with larger delays  the algorithms are convergent with nominal stepsizes while the theoretical stepsize may become too small resulting in slow convergence in practice. It would be interesting to study if the stepsize conditions presented in this paper can be relaxed for the special case when $h_i$, $g_i$ are indicator and quadratic functions.

\section{Conclusions}
We considered the application of primal-dual algorithms for solving structured optimization problems over message-passing architectures. The coupling between agents was classified as total and partial coupling. For each case a separate algorithm was studied and it was shown that the communication delay is tolerated provided that the stepsizes are small enough, and that some strong convexity assumption holds. In the case of total coupling a variant of the proposed algorithm was studied  that allows random and independent activation of the agents. 
Future work consists of extending the convergence analysis to the partially asynchronous framework and exploring Lyapunov functions that allow for nonconvex cost functions.

\begin{figure}
  \center
  \includegraphics[width=0.92\linewidth]{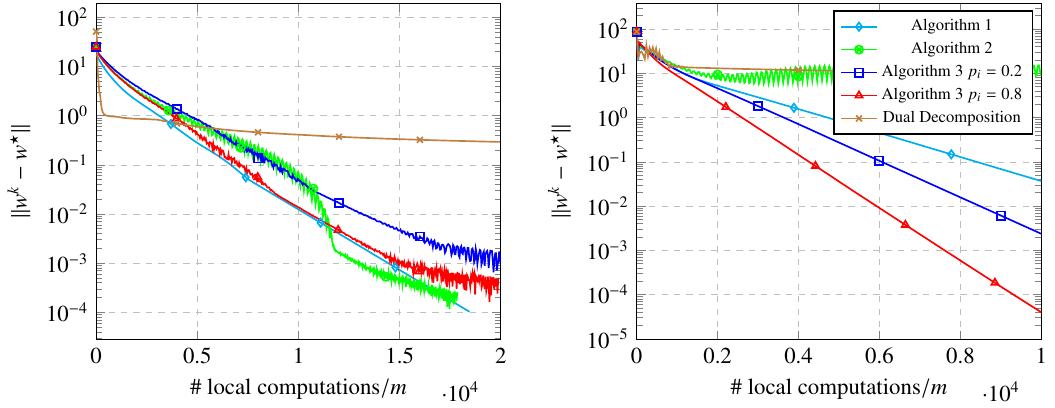}
  \caption{Comparison for the convergence of the algorithms for $m=5$, $B=1$ (left) and $m=50$,  $B=10$ (right). 
  }
  \label{fig2}
\end{figure}

	\ifarxiv
		\bibliographystyle{tfs}
        \bibliography{Bibliography.bib}

        \begin{appendix}
        \section{Ommited lemmas}

          This appendix includes some results that are omitted in the main body of the text. 
          \begin{lem}
          	Let $q:\R^n\to \Rinf$ be a proper closed $\mu$-convex function for some $\mu\geq0$. For all $r\in\R^n$, $\omega\in\R^n$ and $\omega_\rho\coloneqq \prox_{\rho q}(\omega)$ the following holds
          \begin{align}
          q(r) -q(\omega_\rho)  
          \geq   \tfrac{1}{\rho}\langle \omega-\omega_\rho,r-\omega_\rho\rangle+ \tfrac{\mu}{2}\|r-\omega_\rho\|^2.
          \label{eq:strconvexity:RandAsyn}
          \end{align}  
          \end{lem} 
          \begin{proof}
          	The inequality follows immediately from the definition of strong convexity and the characterization of proximal mapping \cite[Prop. 16.44]{Bauschke2017Convex}.  
          \end{proof}
          For all   $a,b,c\in\R^n$ and all positive definite matrices $V\in\R^{n\times n}$ the following elementary equality holds.  
          \begin{equation}\label{eq:basic-equality:RandAsyn}
          2\langle a-b,c-b\rangle_V = \|a-b\|_V^2 + \|c-b\|_V^2 -\|a-c\|_V^2. 
          \end{equation} 
           We also make use of the following inequality:
          \begin{equation}\label{eq:simpleYonug:RandAsyn}
            \langle x,y\rangle \leq \tfrac{\varepsilon}{2}\|x\|^2+\tfrac{1}{2\varepsilon}\|y\|^2,\quad \forall x,y\in\R^n, \epsilon>0. 
          \end{equation}
           \Cref{lem:delays:RandAsyn} provides a basic inequality which is crucial in our analysis.  
          Refer to \cite[Chap. 7.5]{Bertsekas1989Parallel} and \cite[Lem. 4]{Zhou2018Distributed} 
          for the proof.  
          \begin{lem}\label{lem:delays:RandAsyn} Let \Cref{ass:delays:RandAsyn} hold. Consider a vector $w^k=(w_1^k,\ldots,w^k_m)$ and its outdated version $w^k[i]$, \textit{cf.} \eqref{eq:outdated:RandAsyn}. Then, the following inequality holds
             \begin{align} 
             \|w^k-w^k[i]\| \leq &\sum_{\tau=[k-B]_+}^{k-1} \|w^{\tau+1}-w^\tau\|.\label{eq:lem:delays-1:RandAsyn} 
             \end{align}
          \end{lem}
          \begin{lem} \label{lem:delayineq:RandAsyn} 
          Suppose that $\mu_h^i,\mu_g^i>0$, $i=1,\ldots,m$, and in the case of \cref{lem:delayineq-1:RandAsyn} let \Cref{ass:optProb-3:RandAsyn} hold.   
           Then, the following hold for any positive constants $\epsilon_1$, $\epsilon_2$, $\epsilon_3$, nonnegative integer $q$ and generic vectors $v=(v_1,\ldots,v_m)$, $y=(y_1,\ldots,y_m)$ with $v_i\in\R^{n_i}$ and $y_i\in\R^{r_i}$:
          	\begin{enumerate}
          		\item \label{lem:delayineq-1:RandAsyn} 
          		\(\sum_{i=1}^{m}\langle \nabla_{i}f(x^{k}[i])-\nabla_{i}f(x^{k}), x_i^\star-v_{i}\rangle \) \noindent
          	\(\displaystyle
          	\!\leq \!\tfrac{\epsilon_1}{2}\|v-x^\star\|^{2}_{M_g}\!+\tfrac{B}{2\epsilon_1}\|\bar{\beta}\|^2_{M_g^{-1}}\sw{1}{1}{x}
          \)	 
          	\item \label{lem:delayineq-3:RandAsyn} \(
          	\sum_{i=1}^m\langle L_{\mydot i}^{\top}\left(u^{k}[i]-u^{k+q}\right),x_{i}^{\star}-v_i\rangle \) 
          	\( \!\leq{}\! 
          	\tfrac{\epsilon_2}{2}\|v-x^{\star}\|^{2}_{M_g}\!+\tfrac{C_{\textrm s}(B+q)}{2\epsilon_2}\sw{1}{1-q}{u}
          	\)
          	\item \label{lem:delayineq-4:RandAsyn} \(
          	\sum_{i=1}^m\langle L_{i\mydot}(x^{k}[i]-x^{k+q}),y_i-u_{i}^{\star}\rangle\)  
          	\( \displaystyle \!\leq{}\! 
          	\tfrac{\epsilon_3}{2}\|y-u^{\star}\|^{2}_{M_h}\!+\tfrac{R_{\textrm s}(B+q)}{2\epsilon_3}\sw{1}{1-q}{x}
          	\)
          	\end{enumerate}  
          	 where $R_{\rm s}$, $C_{\rm s}$ are defined in \eqref{eq:RsCs:RandAsyn}, $M_g,M_h$ in \eqref{eq:conv-diag:RandAsyn} and \eqref{eq:M:RandAsyn}. 
          \end{lem}
          \begin{proof}
          	We provide the proof for the first inequality and omit the rest noting that they are derived following a similar argument. Using the Cauchy–Schwarz inequality we have 
          	\begin{align*}
          		\sum_{i=1}^{m}\langle x^\star_i-v_{i},\nabla_{i}f(x^{k}[i])-\nabla_{i}f(x^{k})\rangle  {}& {}\leq \sum_{i=1}^{m}\|v_{i}-x^\star_i\|\|\nabla_{i}f(x^{k})-\nabla_{i}f(x^{k}[i])\|\nonumber\\
           &{} \overset{\eqref{eq:Lipz:RandAsyn}}{\leq}{}\sum_{i=1}^{m}\bar{\beta}_i\|v_{i}-x^\star_i\|\|x^{k}-x^{k}[i]\|\nonumber\\
           & {}\overset{\eqref{eq:lem:delays-1:RandAsyn}}{\leq}{}\sum_{i=1}^{m}\bar{\beta}_i\|v_{i}-x^\star_i\|\Big(\sum_{\tau=[k-B]_{+}}^{k-1}\|x^{\tau+1}-x^{\tau}\|\Big)\nonumber\\
           &{} ={}\sum_{i=1\vphantom{B)}}^{m}\sum_{\tau=[k-B]_{+}}^{k-1}\bar{\beta}_i\|v_{i}-x^\star_i\|\|x^{\tau+1}-x^{\tau}\|\nonumber\\
           \dueto{\eqref{eq:simpleYonug:RandAsyn} with $\varepsilon=\tfrac{\mu_g^i\epsilon_1}{B}$}
           & \leq\tfrac{1}{2}\sum_{i=1\vphantom{B)}}^{m}\sum_{\tau=[k-B]_{+}}^{k-1}\!\!\!\!\!\!{\Big(\tfrac{\mu_g^i\epsilon_1}{B}\|v_{i}\!\!\!-x^\star_i\|^{2}\!\!+\tfrac{\bar{\beta}_i^2B}{\mu_g^i\epsilon_1}\|x^{\tau+1}\!\!-x^{\tau}\|^{2}\Big)}\nonumber\\
           & \leq\tfrac{\epsilon_1}{2}\|v-x^\star\|^{2}_{M_g}+\tfrac{B}{2\epsilon_1}\|\bar{\beta}\|^2_{M_g^{-1}}\sum_{\mathclap{\tau=[k-B]_{+}}}^{k-1}\|x^{\tau+1}-x^{\tau}\|^{2},
          \end{align*}
          proving the claim. 
          \end{proof}

        \end{appendix}
	\else
		\phantomsection
		\addcontentsline{toc}{section}{References}
        \bibliographystyle{tfs}
        \bibliography{Bibliography.bib}

\begin{thebibliography}{10}
\providecommand{\MR}{\relax\unskip\space MR }
\providecommand{\url}[1]{\normalfont{#1}}
\providecommand{\urlprefix}{Available at }

\bibitem{Agarwal2011Distributed}
A. Agarwal and J.C. Duchi, \emph{Distributed delayed stochastic optimization},
  in \emph{Advances in Neural Information Processing Systems 24},  2011, pp.
  873--881.

\bibitem{Arrow1958Studies}
K.J. Arrow, L. Hurwicz, and H. Uzawa, \emph{Studies in linear and non-linear
  programming}, Stanford University Press: Stanford, 1958.

\bibitem{Aytekin2016Analysis}
A. Aytekin, H.R. Feyzmahdavian, and M. Johansson, \emph{Analysis and
  implementation of an asynchronous optimization algorithm for the parameter
  server}, arXiv:1610.05507  (2016).

\bibitem{Bauschke2017Convex}
H.H. Bauschke and P.L. Combettes, \emph{Convex analysis and monotone operator
  theory in {H}ilbert spaces}, CMS Books in Mathematics, Springer, 2017.

\bibitem{Bertsekas1989Parallel}
D.P. Bertsekas and J.N. Tsitsiklis, \emph{Parallel and distributed computation:
  numerical methods}, Vol.~23, Prentice-Hall, 1989.

\bibitem{Bianchi2016Coordinate}
P. Bianchi, W. Hachem, and F. Iutzeler, \emph{A coordinate descent primal-dual
  algorithm and application to distributed asynchronous optimization}, IEEE
  Transactions on Automatic Control 61 (2016), pp. 2947--2957.

\bibitem{BricenoArias2011monotone}
L.M. Brice{\~n}o-Arias and P.L. Combettes, \emph{A monotone {+} skew splitting
  model for composite monotone inclusions in duality}, SIAM Journal on
  Optimization 21 (2011), pp. 1230--1250.

\bibitem{Cannelli2019Asynchronous}
L. Cannelli, F. Facchinei, V. Kungurtsev, and G. Scutari, \emph{Asynchronous
  parallel algorithms for nonconvex optimization}, Mathematical Programming
  (2019).

\bibitem{Chang2016Asynchronous}
T. {Chang}, M. {Hong}, W. {Liao}, and X. {Wang}, \emph{Asynchronous distributed
  {ADMM} for large-scale optimization—part i: Algorithm and convergence
  analysis}, IEEE Transactions on Signal Processing 64 (2016), pp. 3118--3130.

\bibitem{Chen1997Convergence}
G. Chen and R. Rockafellar, \emph{Convergence rates in forward--backward
  splitting}, SIAM Journal on Optimization 7 (1997), pp. 421--444.

\bibitem{Combettes2012Primaldual}
P.L. Combettes and J.C. Pesquet, \emph{Primal-dual splitting algorithm for
  solving inclusions with mixtures of composite, {L}ipschitzian, and
  parallel-sum type monotone operators}, Set-Valued and variational analysis 20
  (2012), pp. 307--330.

\bibitem{Combettes2001QuasiFejerian}
P.L. Combettes, \emph{Quasi-{F}ej{\'e}rian analysis of some optimization
  algorithms}, Studies in Computational Mathematics 8 (2001), pp. 115--152.

\bibitem{Combettes2015Stochastic}
P.L. Combettes and J.C. Pesquet, \emph{Stochastic quasi-{F}ej{\'e}r
  block-coordinate fixed point iterations with random sweeping}, SIAM Journal
  on Optimization 25 (2015), pp. 1221--1248.

\bibitem{Condat2013primaldual}
L. Condat, \emph{A primal-dual splitting method for convex optimization
  involving {L}ipschitzian, proximable and linear composite terms}, Journal of
  Optimization Theory and Applications 158 (2013), pp. 460--479.

\bibitem{Drori2015simple}
Y. Drori, S. Sabach, and M. Teboulle, \emph{A simple algorithm for a class of
  nonsmooth convex-concave saddle-point problems}, Operations Research Letters
  43 (2015), pp. 209--214.

\bibitem{Duchi2012Dual}
J.C. Duchi, A. Agarwal, and M.J. Wainwright, \emph{Dual averaging for
  distributed optimization: convergence analysis and network scaling}, IEEE
  Transactions on Automatic control 57 (2012), pp. 592--606.

\bibitem{Feyzmahdavian2015asynchronous}
H.R. {Feyzmahdavian}, A. {Aytekin}, and M. {Johansson}, \emph{An asynchronous
  mini-batch algorithm for regularized stochastic optimization}, IEEE
  Transactions on Automatic Control 61 (2016), pp. 3740--3754.

\bibitem{Hale2017Asynchronous}
M.T. Hale, A. Nedić, and M. Egerstedt, \emph{Asynchronous multiagent
  primal-dual optimization}, IEEE Transactions on Automatic Control 62 (2017),
  pp. 4421--4435.

\bibitem{Iutzeler2013Asynchronous}
F. Iutzeler, P. Bianchi, P. Ciblat, and W. Hachem, \emph{Asynchronous
  distributed optimization using a randomized alternating direction method of
  multipliers}, in \emph{52nd IEEE Conference on Decision and Control (CDC)}.
  2013, pp. 3671--3676.

\bibitem{Johansson2010Randomized}
B. Johansson, M. Rabi, and M. Johansson, \emph{A randomized incremental
  subgradient method for distributed optimization in networked systems}, SIAM
  Journal on Optimization 20 (2010), pp. 1157--1170.

\bibitem{Latafat2019Randomized}
P. Latafat, N.M. Freris, and P. Patrinos, \emph{A new randomized
  block-coordinate primal-dual proximal algorithm for distributed
  optimization}, IEEE Transactions on Automatic Control 64 (2019), pp.
  4050--4065.

\bibitem{Latafat2017Asymmetric}
P. Latafat and P. Patrinos, \emph{Asymmetric forward--backward--adjoint
  splitting for solving monotone inclusions involving three operators},
  Computational Optimization and Applications 68 (2017), pp. 57--93.

\bibitem{Latafat2018PrimalDual}
P. Latafat and P. Patrinos, \emph{Primal-dual proximal algorithms for
  structured convex optimization: A unifying framework}, in \emph{Large-Scale
  and Distributed Optimization}, P. Giselsson and A. Rantzer, eds., Springer
  International Publishing,  2018, pp. 97--120.

\bibitem{Latafat2016primaldual}
P. Latafat, L. Stella, and P. Patrinos, \emph{New primal-dual proximal
  algorithm for distributed optimization}, in \emph{55th IEEE Conference on
  Decision and Control (CDC)}, Dec. 2016, pp. 1959--1964.

\bibitem{Lin2016Distributed}
P. Lin, W. Ren, and Y. Song, \emph{Distributed multi-agent optimization subject
  to nonidentical constraints and communication delays}, Automatica 65 (2016),
  pp. 120 -- 131.

\bibitem{Liu2015Asynchronous}
J. Liu and S.J. Wright, \emph{Asynchronous stochastic coordinate descent:
  Parallelism and convergence properties}, SIAM Journal on Optimization 25
  (2015), pp. 351--376.

\bibitem{Lobel2011Distributed}
I. Lobel, A. Ozdaglar, and D. Feijer, \emph{Distributed multi-agent
  optimization with state-dependent communication}, Mathematical programming
  129 (2011), pp. 255--284.

\bibitem{Nedic2001381Distributed}
A. Nedi{\'c}, D. Bertsekas, and V. Borkar, \emph{Distributed asynchronous
  incremental subgradient methods}, in \emph{Inherently Parallel Algorithms in
  Feasibility and Optimization and their Applications}, D. Butnariu, Y. Censor,
  and S. Reich, eds., Studies in Computational Mathematics Vol.~8, Elsevier,
  2001, pp. 381 -- 407.

\bibitem{Nedic2009Subgradient}
A. Nedi{\'{c}} and A. Ozdaglar, \emph{Subgradient methods for saddle-point
  problems}, Journal of Optimization Theory and Applications 142 (2009), pp.
  205--228.

\bibitem{Nedic2009Distributed}
A. Nedi{\'c} and A. Ozdaglar, \emph{Distributed subgradient methods for
  multi-agent optimization}, IEEE Transactions on Automatic Control 54 (2009),
  pp. 48--61.

\bibitem{Peng2016ARock}
Z. Peng, Y. Xu, M. Yan, and W. Yin, \emph{{ARock}: An algorithmic framework for
  asynchronous parallel coordinate updates}, SIAM Journal on Scientific
  Computing 38 (2016), pp. A2851--A2879.

\bibitem{Pesquet2015class}
J.C. Pesquet and A. Repetti, \emph{A class of randomized primal-dual algorithms
  for distributed optimization}, Journal of Nonlinear and Convex Analysis 16
  (2015), pp. 2453--2490.

\bibitem{Raffard2004Distributed}
R.L. Raffard, C.J. Tomlin, and S.P. Boyd, \emph{Distributed optimization for
  cooperative agents: application to formation flight}, in \emph{2004 43rd IEEE
  Conference on Decision and Control (CDC)}, Vol.~3, Dec. 2004, pp. 2453--2459.

\bibitem{Robbins1985convergence}
H. Robbins and D. Siegmund, \emph{A convergence theorem for non negative almost
  supermartingales and some applications}, in \emph{Herbert Robbins Selected
  Papers}, Springer,  1985, pp. 111--135.

\bibitem{Rockafellar1970Convex}
R.T. Rockafellar, \emph{Convex analysis}, Princeton University Press, 1970.

\bibitem{Schouwenaars2004Decentralized}
T. Schouwenaars, J. How, and E. Feron, \emph{Decentralized cooperative
  trajectory planning of multiple aircraft with hard safety guarantees}, in
  \emph{AIAA Guidance, Navigation, and Control Conference and Exhibit}. 2004,
  pp. 1--14.

\bibitem{Shi2015Proximal}
W. {Shi}, Q. {Ling}, G. {Wu}, and W. {Yin}, \emph{A proximal gradient algorithm
  for decentralized composite optimization}, IEEE Transactions on Signal
  Processing 63 (2015), pp. 6013--6023.

\bibitem{Terelius2011Decentralized}
H. Terelius, U. Topcu, and R.M. Murray, \emph{Decentralized multi-agent
  optimization via dual decomposition}, IFAC Proceedings Volumes 44 (2011), pp.
  11245 -- 11251. 18th IFAC World Congress.

\bibitem{Paul1991Rate}
P. Tseng, \emph{On the rate of convergence of a partially asynchronous gradient
  projection algorithm}, SIAM Journal on Optimization 1 (1991), pp. 603--619.

\bibitem{Tsianos2012Distributed}
K.I. {Tsianos} and M.G. {Rabbat}, \emph{Distributed dual averaging for convex
  optimization under communication delays}, in \emph{2012 American Control
  Conference (ACC)}, June. 2012, pp. 1067--1072.

\bibitem{Vu2013splitting}
B.C. V{\~u}, \emph{A splitting algorithm for dual monotone inclusions involving
  cocoercive operators}, Advances in Computational Mathematics 38 (2013), pp.
  667--681.

\bibitem{Wang2015Cooperative}
H. {Wang}, X. {Liao}, T. {Huang}, and C. {Li}, \emph{Cooperative distributed
  optimization in multiagent networks with delays}, IEEE Transactions on
  Systems, Man, and Cybernetics: Systems 45 (2015), pp. 363--369.

\bibitem{Wu2018Decentralized}
T. {Wu}, K. {Yuan}, Q. {Ling}, W. {Yin}, and A.H. {Sayed}, \emph{Decentralized
  consensus optimization with asynchrony and delays}, IEEE Transactions on
  Signal and Information Processing over Networks 4 (2018), pp. 293--307.

\bibitem{Zhang2014Asynchronous}
R. Zhang and J.T. Kwok, \emph{Asynchronous Distributed {ADMM} for Consensus
  Optimization}, in \emph{Proceedings of the 31st International Conference on
  International Conference on Machine Learning}. 2014, pp. 1701--1709.

\bibitem{Zhou2018Distributed}
Y. Zhou, Y. Liang, Y. Yu, W. Dai, and E.P. Xing, \emph{Distributed proximal
  gradient algorithm for partially asynchronous computer clusters}, Journal of
  Machine Learning Research 19 (2018), pp. 1--32.

\end{thebibliography}
        \appendix
        \section{Ommited lemmas}

        This appendix includes some results that are omitted in the main body of the text. 
        \begin{lem}
        	Let $q:\R^n\to \Rinf$ be a proper closed $\mu$-convex function for some $\mu\geq0$. For all $r\in\R^n$, $\omega\in\R^n$ and $\omega_\rho\coloneqq \prox_{\rho q}(\omega)$ the following holds
        \begin{align}
        q(r) -q(\omega_\rho)  
        \geq   \tfrac{1}{\rho}\langle \omega-\omega_\rho,r-\omega_\rho\rangle+ \tfrac{\mu}{2}\|r-\omega_\rho\|^2.
        \label{eq:strconvexity:RandAsyn}
        \end{align}  
        \end{lem} 
        \begin{proof}
        	The inequality follows immediately from the definition of strong convexity and the characterization of proximal mapping \cite[Prop. 16.44]{Bauschke2017Convex}.  
        \end{proof}
        For all   $a,b,c\in\R^n$ and all positive definite matrices $V\in\R^{n\times n}$ the following elementary equality holds.  
        \begin{equation}\label{eq:basic-equality:RandAsyn}
        2\langle a-b,c-b\rangle_V = \|a-b\|_V^2 + \|c-b\|_V^2 -\|a-c\|_V^2. 
        \end{equation} 
         We also make use of the following inequality:
        \begin{equation}\label{eq:simpleYonug:RandAsyn}
          \langle x,y\rangle \leq \tfrac{\varepsilon}{2}\|x\|^2+\tfrac{1}{2\varepsilon}\|y\|^2,\quad \forall x,y\in\R^n, \epsilon>0. 
        \end{equation}
         \Cref{lem:delays:RandAsyn} provides a basic inequality which is crucial in our analysis.  
        Refer to \cite[Chap. 7.5]{Bertsekas1989Parallel} and \cite[Lem. 4]{Zhou2018Distributed} 
        for the proof.  
        \begin{lem}\label{lem:delays:RandAsyn} Let \Cref{ass:delays:RandAsyn} hold. Consider a vector $w^k=(w_1^k,\ldots,w^k_m)$ and its outdated version $w^k[i]$, \textit{cf.} \eqref{eq:outdated:RandAsyn}. Then, the following inequality holds
           \begin{align} 
           \|w^k-w^k[i]\| \leq &\sum_{\tau=[k-B]_+}^{k-1} \|w^{\tau+1}-w^\tau\|.\label{eq:lem:delays-1:RandAsyn} 
           \end{align}
        \end{lem}
        \begin{lem} \label{lem:delayineq:RandAsyn} 
        Suppose that $\mu_h^i,\mu_g^i>0$, $i=1,\ldots,m$, and in the case of \cref{lem:delayineq-1:RandAsyn} let \Cref{ass:optProb-3:RandAsyn} hold.   
         Then, the following hold for any positive constants $\epsilon_1$, $\epsilon_2$, $\epsilon_3$, nonnegative integer $q$ and generic vectors $v=(v_1,\ldots,v_m)$, $y=(y_1,\ldots,y_m)$ with $v_i\in\R^{n_i}$ and $y_i\in\R^{r_i}$:
        	\begin{enumerate}
        		\item \label{lem:delayineq-1:RandAsyn} 
        		\(\sum_{i=1}^{m}\langle \nabla_{i}f(x^{k}[i])-\nabla_{i}f(x^{k}), x_i^\star-v_{i}\rangle \) \noindent
        	\(\displaystyle
        	\!\leq \!\tfrac{\epsilon_1}{2}\|v-x^\star\|^{2}_{M_g}\!+\tfrac{B}{2\epsilon_1}\|\bar{\beta}\|^2_{M_g^{-1}}\sw{1}{1}{x}
        \)	 
        	\item \label{lem:delayineq-3:RandAsyn} \(
        	\sum_{i=1}^m\langle L_{\mydot i}^{\top}\left(u^{k}[i]-u^{k+q}\right),x_{i}^{\star}-v_i\rangle \) 
        	\( \!\leq{}\! 
        	\tfrac{\epsilon_2}{2}\|v-x^{\star}\|^{2}_{M_g}\!+\tfrac{C_{\textrm s}(B+q)}{2\epsilon_2}\sw{1}{1-q}{u}
        	\)
        	\item \label{lem:delayineq-4:RandAsyn} \(
        	\sum_{i=1}^m\langle L_{i\mydot}(x^{k}[i]-x^{k+q}),y_i-u_{i}^{\star}\rangle\)  
        	\( \displaystyle \!\leq{}\! 
        	\tfrac{\epsilon_3}{2}\|y-u^{\star}\|^{2}_{M_h}\!+\tfrac{R_{\textrm s}(B+q)}{2\epsilon_3}\sw{1}{1-q}{x}
        	\)
        	\end{enumerate}  
        	 where $R_{\rm s}$, $C_{\rm s}$ are defined in \eqref{eq:RsCs:RandAsyn}, $M_g,M_h$ in \eqref{eq:conv-diag:RandAsyn} and \eqref{eq:M:RandAsyn}. 
        \end{lem}
        \begin{proof}
        	We provide the proof for the first inequality and omit the rest noting that they are derived following a similar argument. Using the Cauchy–Schwarz inequality we have 
        	\begin{align*}
        		\sum_{i=1}^{m}\langle x^\star_i-v_{i},\nabla_{i}f(x^{k}[i])-\nabla_{i}f(x^{k})\rangle  {}& {}\leq \sum_{i=1}^{m}\|v_{i}-x^\star_i\|\|\nabla_{i}f(x^{k})-\nabla_{i}f(x^{k}[i])\|\nonumber\\
         &{} \overset{\eqref{eq:Lipz:RandAsyn}}{\leq}{}\sum_{i=1}^{m}\bar{\beta}_i\|v_{i}-x^\star_i\|\|x^{k}-x^{k}[i]\|\nonumber\\
         & {}\overset{\eqref{eq:lem:delays-1:RandAsyn}}{\leq}{}\sum_{i=1}^{m}\bar{\beta}_i\|v_{i}-x^\star_i\|\Big(\sum_{\tau=[k-B]_{+}}^{k-1}\|x^{\tau+1}-x^{\tau}\|\Big)\nonumber\\
         &{} ={}\sum_{i=1\vphantom{B)}}^{m}\sum_{\tau=[k-B]_{+}}^{k-1}\bar{\beta}_i\|v_{i}-x^\star_i\|\|x^{\tau+1}-x^{\tau}\|\nonumber\\
         \dueto{\eqref{eq:simpleYonug:RandAsyn} with $\varepsilon=\tfrac{\mu_g^i\epsilon_1}{B}$}
         & \leq\tfrac{1}{2}\sum_{i=1\vphantom{B)}}^{m}\sum_{\tau=[k-B]_{+}}^{k-1}\!\!\!\!\!\!{\Big(\tfrac{\mu_g^i\epsilon_1}{B}\|v_{i}\!\!\!-x^\star_i\|^{2}\!\!+\tfrac{\bar{\beta}_i^2B}{\mu_g^i\epsilon_1}\|x^{\tau+1}\!\!-x^{\tau}\|^{2}\Big)}\nonumber\\
         & \leq\tfrac{\epsilon_1}{2}\|v-x^\star\|^{2}_{M_g}+\tfrac{B}{2\epsilon_1}\|\bar{\beta}\|^2_{M_g^{-1}}\sum_{\mathclap{\tau=[k-B]_{+}}}^{k-1}\|x^{\tau+1}-x^{\tau}\|^{2},
        \end{align*}
        proving the claim. 
        \end{proof}
	\fi

\end{document}